\documentclass[12pt,a4paper]{amsart}
\usepackage[utf8]{inputenc}
\usepackage[sc,osf]{mathpazo}
\usepackage[scaled=0.85]{beramono}
\usepackage[euler-digits]{eulervm}
\usepackage{amssymb}
\usepackage{mathtools}
\usepackage[a4paper,centering]{geometry}
\usepackage[pdfdisplaydoctitle,colorlinks,breaklinks,urlcolor=blue,linkcolor=blue,citecolor=blue]{hyperref}
\usepackage[capitalise,nameinlink]{cleveref}
\crefname{equation}{}{}
\usepackage{enumitem}
\newtheorem{theorem}{Theorem}[section]
\newtheorem{lemma}[theorem]{Lemma}
\newtheorem{proposition}[theorem]{Proposition}
\newtheorem{corollary}[theorem]{Corollary}

\newtheorem{conjecture*}{Conjecture}
\theoremstyle{definition}
\newtheorem{definition}[theorem]{Definition}

\theoremstyle{remark}
\newtheorem{remark}[theorem]{Remark}
\numberwithin{equation}{section}
\newcommand{\C}{\mathbb{C}}
\newcommand{\E}{\mathbb{E}}
\newcommand{\N}{\mathbb{N}}
\newcommand{\R}{\mathbb{R}}
\newcommand{\Z}{\mathbb{Z}}
\newcommand{\Tb}{\mathbb{T}}
\newcommand{\Pb}{\mathbb{P}}

\newcommand{\Dc}{\mathcal{D}}
\newcommand{\Lc}{\mathcal{L}}
\newcommand{\Pc}{\mathcal{P}}
\newcommand{\Vc}{\mathcal{V}}
\newcommand{\leb}[1]{{\mathcal{L}(#1)}} 

\usepackage{scalerel}
\usepackage{stackengine}
\def\avint{%
  \,\ThisStyle{%
    \ensurestackMath{%
      \stackinset{c}{0\LMpt}{c}{0\LMpt}{\SavedStyle-}{\SavedStyle\phantom{\int}}
    }%
    \setbox0=\hbox{$\SavedStyle\int\,$}\kern-\wd0
  }%
  \int
}
\newcommand{\Lnorm}{L_\text{\rm norm}}

\usepackage{tcolorbox}
\DeclareMathOperator{\lap}{\Delta}
\DeclareMathOperator{\Var}{Var}
\DeclareMathOperator{\sg}{\cdot\nabla}
\newcommand{\norm}[2]{\|#1\|_{#2}}
\newcommand{\semi}[2]{[\![#1]\!]_{#2}}

\usepackage{mathscinet}
\def\MRnum#1\empty{#1}
\renewcommand{\MRhref}[2]{%
  \href{http://www.ams.org/mathscinet-getitem?mr=#1}{#2}
}
\renewcommand{\MR}[1]{
  \relax\ifhmode\unskip\space\fi
  \MRhref{\MRnum#1\empty}{\texttt{\Tiny[MR\MRnum#1\empty]}}
}

\begin{document}
\title[Instantaneous total enhanced dissipation]{Instantaneous total enhanced dissipation for very rough shear flows}
\author[M. Romito]{Marco Romito}
\author[L. Roveri]{Leonardo Roveri}
\address{Dipartimento di Matematica, Universit\`a di Pisa, Largo Bruno Pontecorvo 5, 56127 Pisa, Italia }
\email{\href{mailto:marco.romito@unipi.it}{marco.romito@unipi.it}}
\urladdr{\url{http://people.dm.unipi.it/romito}}
\email{\href{mailto:leonardo.roveri@phd.unipi.it}{leonardo.roveri@phd.unipi.it}}
\begin{abstract}
    This paper investigates enhanced dissipation for a passive scalar advected by "very rough" horizontal shear flows, described by an advection-diffusion equation on the 2D torus. The authors extend results of Galeati and Gubinelli \cite{GalGub2023} to generic flows in negative Besov spaces, proving that the dissipation rate increases to infinity as viscosity vanishes. This is obtained by deriving (non-sharp) upper and lower bounds on the dissipation rate. The upper bound holds for truly irregular velocities, namely those verifying a suitable version of the Wei irregularity index \cite{Wei2021}. As a by-product, it follows that for truly rough shear flows the vanishing viscosity solution to the corresponding inviscid equation is trivial.
\end{abstract}
\maketitle


\section{Introduction}

With this work we provide a study on the dissipation properties of an advection-diffusion equation on the two-dimensional torus,
\[
    \begin{cases}
        \partial_t f + u \sg f = \nu \lap f, \\
        f|_{t=0} = f_0,                      \\
        \int_\Tb f(t,x,y)\,dx = 0 .
    \end{cases}
\]
Here, the scalar $f$, said to be \textit{passive} because it has no feedback effects on the flow, is advected by a rough vector field $u$ and $\nu\ll1$ is a diffusion coefficient.
Equations like the above are classic and have been the subject of an extended literature (cf. \cite{IKX14, LTD11, Seis13} for references), and model for example concentration of chemicals in a fluid.

We consider in particular $u$ to be a horizontal shear flow, hence of the form $\tilde u(x,y) = (u(y),0)$, which allows to reformulate the model as
\begin{equation}\label{eq:gg21}
    \begin{cases}
        \partial_t f + u \partial_x f = \nu \lap f, \\
        f|_{t=0} = f_0,                             \\
        \int_\Tb f(t,x,y)\,dx = 0.
    \end{cases}
\end{equation}
In this case, a closely related problem regards the hypoelliptic version of \cref{eq:gg21}, arising when we change the diffusive term and substitute the Laplacian operator with a second-order derivative in the second direction only, i.e.
\begin{equation}
    \begin{cases}
        \partial_t f + u \partial_x f = \nu \partial_y^2 f, \\
        f|_{t=0} = f_0,                                     \\
        \int_\Tb f(t,x,y)\,dx = 0.
    \end{cases}
\end{equation}
The above is used for example in the study of boundary layers \cite{LWX16}, and was first considered by Kolmogorov in \cite{K34}.

It is well known that energy estimates show how the energy of the system, represented by the $L^2$ norm of $f$, decays in time simply due to the dissipation provided by the Laplacian (or by an analogous second order term).
However, more and more attention has been devoted to studying possible influences of the advection on the diffusion.
In the last two decades, notably following the work \cite{CKRZ08}, many authors have focused on the fact that the stirring provided by the transport term may cause dissipation of energy at a much faster rate than the classic one (which after all corresponds to that of the heat equation), see for example \cite{BGM20, BW13, Zla10}.
This led to define the class of \textit{diffusion enhancing} velocities, namely those for which there exists a a rate $r:(0,1)\to\R$ such that the energy norm of the solution satisfies $\norm{f(t)}{L^2} \lesssim \exp \left( - r(\nu) t \right) \norm{f_0}{L^2}$.

Our work in particular extends the achievements of \cite{Col21} and \cite{GalGub2023} to describe such dissipative effects for a class of advecting velocities as broad as possible.
In \cite{Col21}, enhanced dissipation is proved for a specific Weierstrass-type shear flow $u$, sharply $\alpha$-H\"older for positive $\alpha\in(0,1)$, that is constructed ad hoc.
The rate found therein is of the form $r(\nu) \sim \nu^{\frac{\alpha}{\alpha+2}}$.
However, it is left as an open problem the potential extension of their result to a more \emph{generic} class of velocities.
This matter is addressed in \cite{GalGub2023}, where the authors prove that the same rate is sharp and can be recovered for \emph{any} shear flow in the larger Besov space $B_{1,\infty}^\alpha$, for positive $\alpha\in(0,1)$, provided that an irregularity condition is satisfied.
Following the same path, we push the regularity boundaries in order to allow the velocity to be a distribution.

By \textit{rough flow}, indeed, we mean that $u\in B^\alpha_{1,\infty}$ with $\alpha\in(-\frac12,0)$, so that the shear flow $\tilde{u}$ is not necessarily a function.
The reason behind the restriction $\alpha>-\frac12$ is that, in this range of values of the regularity parameter, one can still make sense of the equation in a classical (weak) way.
However, we expect that the same results obtained here extend to more irregular regimes, as long as a suitable meaning is given to the solution, following for instance the modern approach of singular stochastic equations, see \cite{Hai2014,GubImkPer2015}.

We show that, for negative $\alpha$, the diffusion rate that controls the energy increases to infinity as the viscosity decreases. In particular, we derive explicit lower and upper bounds on this growth.
The lower bound is the key result, as it ensures the rate diverges and depends on the velocity being sufficiently irregular.
The upper bound is also of interest, as it helps to characterise this explosion and applies to any advecting velocity, based on the regularity properties of the chosen Besov spaces.

The other goal we achieve is describing the subset of diffusion-enhancing velocities in the chosen Besov space and proving that they are `sufficiently many'.
This is intended in the sense of \emph{prevalence}, as recently considered in \cite{GalGub2020}.
Interested readers can find a thorough introduction to the notion of prevalence there.
This final result is based on the notion of \emph{$\alpha$-irregularity}, as defined in \cref{def:ourWei}.
It also represents a sufficient condition for velocities to be diffusion-enhancing and a necessary condition for the lower bound to hold.

Our main findings can be summarized as follows.

\begin{theorem}\label{th:main}
    Let $\alpha \in (-\frac12,0)$.
    \begin{enumerate}[label=\roman*)]
        \item $\alpha$-irregularity, in the sense of \cref{def:ourWei}, is a sufficient condition for a velocity $u\in B^\alpha_{1,\infty}$ to be diffusion-enhancing with rate $r$ such that
              \begin{equation}\label{eq:mainupper}
                  \nu^{\frac{\tilde\alpha}{\tilde\alpha + 2}}
                  \lesssim
                  r(\nu)
                  \qquad
                  \forall \bar\alpha < \alpha .
              \end{equation}
              Moreover, if $u\in B^\alpha_{1,\infty}$ is diffusion enhancing with rate $r$,
              \begin{equation}\label{eq:mainlower}
                  r(\nu)
                  \lesssim
                  \nu^{-\frac{1/2-\bar\alpha}{5/2+\bar\alpha}}
                  \qquad
                  \forall \tilde\alpha > \alpha .
              \end{equation}
        \item The subset of $\alpha$-irregular velocities is prevalent in $B^\alpha_{1,\infty}$, hence almost every $u\in B^\alpha_{1,\infty}$ is diffusion-enhancing.
    \end{enumerate}
\end{theorem}

Finally, we point out that our choice of $\alpha$ prevents us from studying a non-viscous version of the problem, considered instead in \cite{Col21} and \cite{GalGub2023} to prove inviscid mixing and study its connection with enhanced diffusion, because the equation now needs the regularising effect of the Laplacian in order to be well-posed.
Indeed, our result goes beyond the novelty of distributional advecting velocities, and it can be seen as a \emph{triviality result} for the inviscid transport equation with a very rough shear velocity.
\cref{th:main} above shows that energy dissipates faster and faster as the viscosity vanishes. In other words, regardless of the initial condition, the vanishing viscosity solution of the inviscid transport equation with a very rough advecting velocity is the zero solution.

Our work is organized as follows. \cref{sec:mainresult} contains formal definitions of the main concepts we use, namely diffusion enhancement and the $\alpha$-irregularity condition.
Afterwards, \cref{sec:enhdiff} presents the proofs of the bounds on the diffusion rate, \cref{sec:Wei} discusses in details the irregularity condition, and \cref{sec:prevalence} shows how ``almost every'' element of $B_{1,\infty}^\alpha$ satisfies them.


\section{Preliminaries}\label{sec:mainresult}

Let us first state a formal definition of \textit{enhanced dissipation}.
To this end, we already discussed how it is convenient to consider the hypoelliptic counterpart of our problem,
\begin{equation}\label{eq:hypoelliptic}
    \partial_t f + u \partial_x f = \nu \partial_y^2 f,
\end{equation}
and particularly its Fourier-transformed version
\begin{equation}\label{eq:hypofourier}
    \partial_t f^k + ik u f^k = \nu \partial_y^2 f^k,
\end{equation}
where, for $k\in \Z$, $f^k$ is the $k$-th Fourier coefficient of $f$ with respect to the horizontal variable $x$.
We shall denote by $e^{t(-\imath k u + \nu \partial^2_y)}$ the solution operator of the above equation, for every fixed $k$.
Notice also that one can restrict to the case $k=\pm1$, up to a rescaling of $t$ and $\nu$.

The reason behind the hypoelliptic choice leading to \cref{eq:hypofourier} is that it is easier to treat and its solutions have the same properties as those of \cref{eq:gg21}.
Indeed, consider the Fourier transform in the $x$-variable of the latter, i.e.
\begin{equation}\label{eq:gg21fourier}
    \partial_t f^k + \imath k u f^k = \nu (-k^2 + \partial^2_y) f^k :
\end{equation}
$f^k(t,y)$ solves \cref{eq:gg21fourier} if and only if $g^k(t,y) \coloneqq e^{-k^2t} f^k(t,y)$ solves equation \cref{eq:hypofourier}.
Moreover, such $g^k$ would have the same spacial regularity as $f^k$.
We refer to \cref{app:wellposedness} for a proof of well-posedness of \cref{eq:hypofourier}.

The next definition clarifies what we intend for \textit{enhanced dissipation}, in terms of behaviour of solutions to \cref{eq:gg21} at each Fourier level $f^k$.

\begin{definition}\label{def:enhdiss}
    A velocity field $u$ is said to be \emph{diffusion enhancing} on $L^2(\Tb;\C)$ with rate $r$ if there exists a constant $C>0$ such that
    \begin{equation}\label{eq:enhdiss}
        \norm{e^{t(-\imath k u + \nu \partial^2_y)}}{L^2\rightarrow L^2}
        \leq C \exp\bigl( - r\bigl(\tfrac{\nu}{|k|}\bigr)|k| t \bigr),
    \end{equation}
    for all $k\in\Z$, $\nu\in(0,1)$, and $t\geq1$.
\end{definition}

\begin{remark}
    Notice that, since
    \[
        \mathrm{Re}\int_\Tb i k u f^k \overline{f^k} \,dy
        = 0,
    \]
    by a simple energy estimate every $u$ is diffusion enhancing with rate $r(\nu)=\nu$, as long as the problem is well posed.
    Our goal is then to describe the optimal diffusion rate.
\end{remark}

We do not expect our result to hold in general for any shear flow.
Indeed, it is the irregularity of the advective velocity that plays a major role.
As in \cite{Wei2021,GalGub2023}, we are going to identify an \emph{irregularity condition} on $u$ that ensures enhanced dissipation. Preliminarly, define an elliptic regularization $U$ of $u$ via
\begin{equation}\label{def:ellipreg}
    \begin{cases}
        - \partial_{y}^2 U
        = u - \bar u,
        \qquad\text{on }\Tb, \\
        \int_\Tb U(y)\,dy
        = 0,                 \\
        U\text{ periodic on }\Tb,
    \end{cases}
\end{equation}
where $\bar u$ is the average of $u$ on $\Tb$. Our version of \emph{Wei's irregularity condition} is given as follows.
\begin{definition}\label{def:ourWei}
    A velocity $u$ is \emph{$\alpha$-irregular} if
    \begin{equation}\label{eq:ourWei}
        \inf_{\substack{\leb{I}\leq 1\\ P\in\Pc_1}}
        \leb{I}^{-\alpha}\Bigl(\avint_I |\partial_y U(y) - P(y)|^2\,dy\Bigr)^{\frac12}
        >0,
    \end{equation}
    where $\leb{I}$ is the Lebesgue measure of an interval $I\subseteq\Tb$ and $\Pc_1$ is the space of polynomials of degree at most $1$.
\end{definition}

\begin{remark}
    The above is much similar to the irregularity condition given in \cite{GalGub2023}, which in turn extends the one found in \cite{Wei2021}.
    With respect to \cite{GalGub2023}, however, we have chosen a specific primitive of $u$: this will allow to exploit the one-to-one relation between the two.
\end{remark}


\section{Enhanced diffusion}\label{sec:enhdiff}

We devote this section to prove the first claim of \cref{th:main}. The proof can be naturally split in two parts: the first covers the upper bound on the dissipation rate, while the second focuses on the lower one.


\subsection{Lower bound}

Proving a lower bound on the dissipation of the energy of solutions,
\[
    \norm{e^{-iktu + \nu t\partial_y^2}}{L^2(\Tb;\C)\rightarrow L^2(\Tb;\C)}
    = \sup_{\substack{f\in L^2(\Tb;\C) \\ \norm{f}{{L^2}}=1}}
    \int_\Tb e^{t(-\imath k u + \nu \partial^2_y)} f \,dy ,
\]
essentially reduces to show that the dissipation rate $r$ can be controlled from above by some function of the diffusion coefficient $\nu$.
In particular, we focus here on \cref{eq:mainlower}, i.e.
\begin{equation}\label{eq:lowbound}
    r(\nu)
    \lesssim \nu^{-\frac{1/2-\bar\alpha}{5/2+\bar\alpha}}
    \qquad
    \forall \ \nu\in(0,1) , \ \bar\alpha < \alpha ,
\end{equation}
with $\alpha\in(-\frac12,0)$ being the regularity coefficient of the Besov space $B_{1,\infty}^\alpha$ to which $u$ belongs.

In order to recover \cref{eq:lowbound}, we consider the hypoelliptic problem \cref{eq:hypofourier} and study its Fourier transform in the first component,
\begin{equation}\label{eq:fourier}
    \partial_t f + \imath\xi u f = \nu \partial_y^2 f,
\end{equation}
with $f:\Tb\to\C$ and $\xi\in\R$.
We then start by recalling two preliminary lemmas from \cite{GalGub2023}.
While the first, \cref{lemma:feynman-kac}, can be used as is, the second, \cref{lemma:continuity}, which is the key to our result, modifies \cite[Lemma 4.3]{GalGub2023} to cover the case of negative $\alpha$.

\begin{lemma}[{\cite[Lemma 4.2]{GalGub2023}}]\label{lemma:feynman-kac}
    Let $u \in L^1(\Tb)$ and $f$ be a solution to \cref{eq:fourier} with initial data $f_0 \in L^2(\Tb;\C)$.
    Given, for any $(t,y) \in [0,\infty)\times\Tb$, the complex random variable
    \begin{equation}\label{eq:feynman-kac}
        Z_t^Y
        = \exp \left( -\imath\xi \int_0^t u \left( y + \sqrt{2\nu}B_s \right) ds \right) f_0 \left( y + \sqrt{2\nu} B_t \right) ,
    \end{equation}
    where $B$ is a standard, real-valued BM,
    the following Lagrangian Fluctuation-Dissipation relation holds:
    \[
        \norm{f_0}{L^2}^2 - \norm{f_t}{L^2}^2
        = \int_\Tb \operatorname{Var} (Z_t^y) dy .
    \]
\end{lemma}

Notice that \cref{eq:feynman-kac} is well-defined thanks to the requirement $u\in L^1$.
However, to be able to exploit the Lemma we are going to need an approximation, guaranteed by the embedding $B_{1,\infty}^\alpha\subset W^{\alpha-\epsilon, 1+\epsilon}$ for all $\epsilon>0$ and by the fact that, in the latter topology, the subset of $C^\infty$ functions is dense.

\begin{lemma}\label{lemma:continuity}
    Let $f_0\in H^1(\Tb;\C)$, $u\in B_{1,\infty}^\alpha(\Tb)$, $\alpha\in(-\frac12,0)$, and $\xi\in\R$. Then the corresponding solution $f$ of \cref{eq:fourier} satisfies
    \begin{align*}
        \norm{f_0}{L^2}^2 - \norm{f_t}{L^2}^2
        \leq C \norm{f_0}{H^1}^2
        \left( \nu t + |\xi|^2 \left( \nu^{-1/2+\bar\alpha} t^{3/2+\bar\alpha+2\epsilon} \right)^{\frac{1}{1+\epsilon}}
        \norm{u}{B_{1,\infty}^\alpha}^2 \right) ,
    \end{align*}
    where $\bar\alpha\in(\frac{1}{2}, \alpha)$ and the constant $C>0$ depend on $\alpha$ and a suitable $\epsilon>0$.
\end{lemma}

\begin{proof}
    Start by assuming that $u$ is smooth.
    Let $B$, $\tilde{B}$ be two independent Brownian motions and let $\tilde{Z}_t^y$, $Z_t^y$ be two random variables defined as in \cref{eq:feynman-kac} and depending respectively on $B$ and $\tilde B$.
    By \cref{lemma:feynman-kac},
    \begin{align*}
        \norm{f_0}{L^2}^2 - \norm{f_t}{L^2}^2
        ={}      &
        \int_\Tb \Var ( Z_t^y ) d y                                                                                                                                                                 \\
        ={}      &
        \frac{1}{2} \int_\Tb \E \left[ \left| Z_t^y - \tilde{Z}_t^y \right|^2 \right] d y                                                                                                           \\
        \leq     &
        \E \left[ \int_\Tb \left| f_0 (y + \sqrt{\nu}B_t) - f_0 (y + \sqrt{\nu}\tilde{B}_t) \right|^2 d y \right]                                                                                   \\
                 & + \norm{f_0}{L^\infty}^2 \E \left[ \int_\Tb \left| e^{\imath \xi \int_0^t u(y+\sqrt{\nu}B_s) d s} - e^{\imath \xi \int_0^t u(y+\sqrt{\nu}\tilde{B}_s) d s} \right|^2 d y \right] \\
        \lesssim &
        \norm{f_0}{H^1}^2 \nu \E \left[ | B_t - \tilde{B}_t |^2 \right]                                                                                                                             \\
                 & + \norm{f_0}{L^\infty}^2 |\xi|^2 \E \left[ \left\|\int_0^t u(\cdot+\sqrt{\nu}B_s) - u(\cdot+\sqrt{\nu}\tilde{B}_s) d s\right\|_{L^2}^2 \right] ,
    \end{align*}
    where in the last inequality we used $|e^{i \xi a} - e^{i \xi b}| \lesssim 1 \wedge |\xi| |a-b|$.
    Now, since for $s\in\R$ and $p,q\in[1,\infty]$ the map $f\mapsto\partial_y f$ is a continuous linear operator from $B_{p,q}^s$ to $B_{p,q}^{s-1}$, we can use a so called \emph{It\^o trick} (see for instance \cite{Fla2011}). In other words, we exploit the regularity of the solution $U$ of
    \begin{equation}\label{eq:itotrick}
        \begin{cases}
            \partial_t U - \frac12\nu\partial_y^2 U = u - \bar u, \\
            U(0) = 0, \quad \int_\Tb U = 0,                       \\
            U\text{ periodic on }\Tb
        \end{cases}
    \end{equation}
    on the torus, where $\nu > 0$ is the diffusion parameter coming from \cref{eq:fourier} and $\bar u$ is the mean of $u$ on $\Tb$.
    Indeed, $U$ is regular enough to apply It\^o formula, yielding
    \begin{align*}
        dU(t-s, y+\sqrt{\nu}B_s)
        = & {}
        \sqrt{\nu}\partial_y U(t-s,y+\sqrt{\nu}B_s) dB_s                                                               \\
          & - \left[\partial_s U(t-s, y+\sqrt{\nu}B_s) - \frac{\nu}{2} \partial_y^2 U(t-s, y+\sqrt{\nu}B_s) \right] ds \\
        = & {}
        \sqrt{\nu}\partial_y U(t-s,y+\sqrt{\nu}B_s) dB_s
        + (\bar u - u(y + \sqrt{\nu}B_s)) d s,
    \end{align*}
    with the second equality simply due to \cref{eq:itotrick}.
    We can then integrate on the time interval $[0,t]$ to obtain (recall that $U(0) = 0$),
    \begin{align*}
        \int_0^t u(y+\sqrt{\nu}B_s) d s
        = U(t,y)
        + t\bar u
        + \sqrt{\nu} \int_0^t \partial_y U (t-s, y + \sqrt{\nu}B_s) d B_s ,
    \end{align*}
    so that
    \begin{align*}
        \E \bigg[ \bigg\|\int_0^t u(\cdot+\sqrt{\nu}B_s) - u(\cdot+\sqrt{\nu} & \tilde{B}_s) d s\bigg\|_{L^2}^2 \bigg] \\
                                                                              & \leq
        2 \nu \int_\Tb \E \left[ \left| \int_0^t \partial_y U(t-s, y + \sqrt{\nu}B_s) d B_s \right|^2 \right] d y      \\
                                                                              & =
        2 \nu \int_\Tb \E \left[ \int_0^t \left| \partial_y U(t-s, y + \sqrt{\nu}B_s) \right|^2 d s \right] d y .
    \end{align*}
    Since $B_{1,\infty}^{\alpha+1}\subseteq H^{\bar\alpha+\frac{1}{2}}$ for all $\bar\alpha<\alpha$, we have that $\partial_y U \in L^2(\Tb)$ and we can rewrite the integral on the right-hand side above as
    \begin{multline*}
        \int_\Tb \E \left[ \int_0^t \left| \partial_y U(t-s, y + \sqrt{\nu}B_s) \right|^2 d s \right] d y
        =
        \E \left[\int_0^t \int_\Tb |\partial_y U(t-s,y) |^2 d y d s \right] \\
        =
        \int_0^t \norm{\partial_y U(s,\cdot)}{L^2}^2 d s .
    \end{multline*}
    Eventually, we obtain a bound of the form
    \[
        \norm{f_0}{L^2}^2 - \norm{f_t}{L^2}^2
        \lesssim
        \norm{f_0}{H^1}^2 \left( \nu t + |\xi|^2 \nu t \int_0^t \norm{\partial_y U(s,\cdot)}{L^2}^2 d s \right) .
    \]
    The next step would be to estimate the integral in time of $\norm{\partial_y U(r,\cdot)}{L^2}$ in terms of $\norm{u}{W^{\alpha-\epsilon, 1+\epsilon}} < \norm{u}{B_{1,\infty}^\alpha}$.
    To this regard, we postpone to \cref{app:parabolic} the proof that, under our assumptions on $\epsilon$ and $\bar\alpha$,
    \begin{equation}\label{eq:Uvsu}
        \int_0^t \norm{\partial_y U(r,\cdot)}{L^2}^2 ds
        \leq
        C \left( \nu^{-\frac32+\bar\alpha-\epsilon}  t^{\frac32+\bar\alpha+2\epsilon} \right)^{\frac{1}{1+\epsilon}} \norm{u}{W^{\alpha-\epsilon,1+\epsilon}}^2
    \end{equation}
    for a suitable choice of $\epsilon$ and $\bar\alpha$, which yields
    \begin{equation}\label{eq:smoothversion}
        \norm{f_0}{L^2}^2 - \norm{f_t}{L^2}^2
        \lesssim
        C_{\alpha,\epsilon} \norm{f_0}{H^1}^2 \left( \nu t + |\xi|^2 \left( \nu^{-1/2+\bar\alpha} t^{5/2+\bar\alpha+3\epsilon} \right)^{\frac{1}{1+\epsilon}} \norm{u}{W^{\alpha-\epsilon, 1+\epsilon}}^2 \right) .
    \end{equation}
    At this point, we can consider a generic $u\in B_{1,\infty}^\alpha$ and approximate it in $W^{\alpha-\epsilon,1+\epsilon}$ with a sequence $\{u^n\}_{n}$ of smooth functions.
    Since this also implies that the corresponding solutions to problem \cref{eq:fourier}, $\{f^n\}_n$, converge in $L^2$ (see \cref{app:wellposedness}), we eventually have that the bound \cref{eq:smoothversion} holds for every $u\in B_{1,\infty}^\alpha$ with the related norm, hence the thesis.
\end{proof}

With \cref{lemma:continuity} at hand, we can now turn to the proof of inequality \cref{eq:lowbound}.

\begin{proposition}
    Let $u\in B_{1,\infty}^\alpha$, $\alpha\in(-\frac{1}{2},0)$.
    If $u$ is diffusion enhancing, in the sense of \cref{def:enhdiss}, then necessarily its diffusion rate $r(\nu)$ satisfies
    \[
        r(\nu)
        \leq
        C \nu^{-1/2+\bar\alpha} t^{5/2+\bar\alpha}
    \]
    for all $\nu\in(0,1)$ and $-\frac12<\bar\alpha<\alpha$, with $C>0$ being a constant depending only on $\bar\alpha$ and $\norm{u}{B_{1,\infty}^\alpha}$.
\end{proposition}

\begin{remark}
    Unfortunately, we have not been able to find a lower bound on the exponent of the dissipation rate that matches the corresponding upper bound.
    \cite{GalGub2023} finds the optimal rate $\frac\alpha{\alpha+2}$ for $\alpha\in(0,1)$, and we expect this same exponent to be optimal for $\alpha\in(-\frac12,0)$.
    However, in our case we cannot directly exploit the equivalent formulation of Besov norms based on the incremental ratio of $u$, like in \cite{GalGub2023}, as it only works for positive $\alpha$.
    Instead, we are forced to go through the It\^o trick and inequality \cref{eq:Uvsu}, which yield a bound with coefficients $( \nu^{-1/2+\bar\alpha} t^{5/2+\bar\alpha+3\epsilon} )^{1/(1+\epsilon)}$ instead of $\nu^{\alpha/2} t^{1+\alpha/2}$.

\end{remark}

\begin{proof}
    Choose, as in \cref{lemma:continuity}, $\epsilon\in(0,\alpha+\sqrt{\alpha^2+4\alpha+2})$ such that $\bar\alpha = \alpha - \epsilon(\frac12 - \alpha + \epsilon)$ belongs to $(-\frac12,\alpha)$ and assume, by contradiction, that
    \[
        \liminf_{\nu\rightarrow0^+} \nu^{\frac{1/2-\bar\alpha}{5/2+\bar\alpha+3\epsilon}} r(\nu) = +\infty .
    \]
    As in \cite{GalGub2023}, one can then choose $f_0\in H^1$, $\norm{f_0}{L^2} = 1$; by \cref{def:enhdiss} and \cref{lemma:continuity}, there exist $C_1,C_2>0$ depending on $\bar\alpha$ such that, for all $\nu\leq1$ and $t\geq1$,
    \begin{align*}
        1 - C_1 e^{-r(\nu)t}
         & \leq
        1 - \norm{e^{-iktu + \nu t\partial_y^2}g_0}{L^2}^2 \\
         & \leq
        C_2 \norm{f_0}{H^1}^2 \left( \nu t + |\xi|^2 \left( \nu^{-1/2+\bar\alpha} t^{3/2+\bar\alpha+2\epsilon} \right)^{\frac{1}{1+\epsilon}} \norm{u}{B_{1,\infty}^\alpha}^2 \right) .
    \end{align*}
    Fix $\xi=1$.
    If $\{\nu_n\}_n$ is a sequence realizing the limit and we consider a sequence of times $\{t_n\}_n$ of the form
    \[
        t_n
        =
        \left( r(\nu_n) \nu_n^{-\frac{1/2-\bar\alpha}{5/2+\bar\alpha+3\epsilon}} \right)^{-\frac{1}{2}} ,
    \]
    then
    \begin{multline}
        1 - C_1 \exp \left( - \left(r(\nu_n)\nu_n^{\frac{1/2-\bar\alpha}{5/2+\bar\alpha+3\epsilon}} \right)^{\frac{1}{2}} \right) \\
        \lesssim
        \left( r(\nu_n) \nu_n^{\frac{1/2-\bar\alpha}{5/2+\bar\alpha+3\epsilon}} \right)^{-\frac{1}{2}} \nu_n^{\frac{3(1+\epsilon)}{5/2+\bar\alpha+3\epsilon}}
        + \left( r(\nu_n) \nu_n^{\frac{1/2-\bar\alpha}{5/2+\bar\alpha+3\epsilon}} \right)^{-\frac{1+\epsilon}{2} \left( \frac{5}{2} + \bar\alpha + 3\epsilon \right)} .
    \end{multline}
    Now, both terms on the right-hand side trivially go to $0$ when $n\rightarrow\infty$, because $5/2+\bar\alpha + 3\epsilon > 0$.
    Since the left-hand side converges instead to $1$, this means that in the limit we have $1\leq0$: therefore, it must hold instead that
    \[
        r(\nu) \leq C \nu^{- \frac{1/2-\bar\alpha}{5/2+\bar\alpha+3\epsilon}} .
    \]
    At this point we can neglect $\epsilon$ and generalise to any $\bar\alpha$ in the interval $(-\frac12,\alpha)$, which is the result we were looking for.
\end{proof}


\subsection{Upper bound}

In the following, we shall use the notion of irregularity provided by $\Lambda(\alpha, m, p, g)$, as defined in \cref{def:ourWei}.
We refer to \cref{sec:Wei} for a generalization and a more detailed description of its properties.

This means that we will found ourselves in the $\alpha$-irregularity regime (see \cref{def:ourWei}).
The key role of such setting is highlighted by \cref{prop:upbound} below, which states a sufficient condition to have an upper bound on the dissipation rate in the case $\alpha<0$.
In particular, such bound is going to have the same form of that found by the authors of \cite{GalGub2023}, hence extending their result to negative regularity Besov spaces.
The reader willing to compare the two works should be advised, however, that the notation does not exactly coincide, as there is a shift in regularity for the $f$ appearing in \cref{def:ourWei}: in \cite{GalGub2023}, $f$ is a derivative of what actually appears in our formulation of $\Lambda$.

In order to prove our result, we need a Lemma contained in \cite{Wei2021}.
The only difference, which does not really impact the proof, is that we have to redefine the quantity $\omega_1$ found there, choosing instead
\[
    \omega_1(\delta,u) \coloneqq \inf_{\bar{y}\in\Tb, P\in\Pc_1} \int_{\bar{y}-\delta}^{\bar{y}+\delta} |\partial_y U(y) - P(y) |^2 d y , \qquad \delta\in (0,1) .
\]
Indeed, we exploit the regularity properties of $\partial_y U$, with $U$ solving
\begin{equation}\label{eq:itotrickII}
    \begin{cases}
        -\partial_y^2 U = u - \bar u, \\
        \int_\Tb U = 0 , \ U \text{ periodic on } \Tb
    \end{cases}
\end{equation}
(where again $\bar u$ is the mean of $u$ on $\Tb$) instead of those of a general primitive of $u$.

\begin{lemma}[{\cite[Lemma 4.3]{Wei2021}}]\label{lemma:b1}
    Let $u \in C(\Tb)$, fix $\nu>0$ and set $F: [0,\infty) \rightarrow[0, \pi/2]$ to be the one-to-one, increasing inverse of $x \mapsto 36 x \tan x$.
    Then, for all $\delta \in(0,1)$ and $t \geq 0$, it holds
    \[
        \norm{e^{t\left(\nu\partial_y^2 - \imath u\right)}}{L^2 \rightarrow L^{2}}
        \leq
        \exp \left( \frac{\pi}{2} - t \nu \delta^{-2} F \left( \delta \nu^{-2} \left( \omega_1 (\delta, u) \right)^2 \right) \right) .
    \]
\end{lemma}

\cref{lemma:b1} is the key to prove the following.

\begin{proposition}\label{prop:upbound}
    For $\alpha\in(-\frac{1}{2},0)$, let $u\in B_{1,\infty}^\alpha$ be such that $\Lambda(\alpha+1, 1, 2, \partial_y U)>0$, with $U$ being the associated solution of \cref{eq:itotrickII}.
    Then $u$ is diffusion enhancing with rate $r(\nu) \gtrsim \nu^{\alpha/(\alpha+2)}$, since
    \begin{equation}\label{eq:upbound}
        \norm{e^{t(\nu\partial_y^2 - \imath ku)}}{L^2\rightarrow L^2}
        \leq
        C_1 \exp \left( C_2 \nu^{\frac{\alpha}{\alpha+2}} |k|^{\frac{\alpha}{\alpha+2}} t \right)
        \quad
        \forall \nu\in(0,1) , k\in\Z_0 , t \geq 0 ,
    \end{equation}
    with the constants $C_1,C_2\geq0$ depending on $\alpha$, $\semi{u}{ B_{1,\infty}^\alpha}$ and $\Lambda(\alpha+1, 1, 2, \partial_y U)$.
\end{proposition}

\begin{proof}
    Without loss of generality, let us consider the case $k=1$.
    Given $u\in B_{1,\infty}^\alpha$, we can take a sequence $\{u^n\}_n$ of smooth functions such that the corresponding solutions $U^n$ of \cref{eq:itotrickII} satisfy $\norm{U^n-U}{L^2} \leq 1/n$, $U$ being the solution of \cref{eq:itotrickII} associated to $u$ (e.g. we can let $u^n\to u$ in a suitable Sobolev space of negative regularity).

    Using the general fact that $a^2\geq b^2/2 - (a-b)^2$, for $\delta\in(0,1)$ it holds
    \begin{multline*}
        \omega_1 (\delta, u^n)
        \geq
        \frac{1}{2} \omega_1 (\delta, u) - \sup_{x \in \Tb} \int_{x-\delta}^{x+\delta} \left|U^n(y)-U(y)\right|^2 d y \\
        \geq
        \frac{1}{2} \omega_1 (\delta, u) - 2 \delta \norm{U^n-U}{L^2}^2
    \end{multline*}
    and, since $\omega_1 (\delta, u) \geq 2^{2 \alpha+3} \delta^{2 \alpha+3} \Lambda(\alpha+1, 1, 2, \partial_y U)^{2}$ by definition, for all $n\in\N$
    \[
        \omega_1 (\delta, u^n)
        \geq
        2^{2 \alpha+2} \delta^{2 \alpha+3} \Lambda(\alpha, 1, 2, u)^2 - 2 \delta \frac{1}{n^2} .
    \]
    By fixing $\nu>0$, $\delta = \nu^{1/(\alpha+2)}$ and applying \cref{lemma:b1} to the continuous functions $u^n$, we can set $C_{1}=e^{\pi / 2}$ and use the monotonicity of $F$ to get
    \[
        \norm{e^{t\left(\partial_y^2-\imath u^n\right)}}{L^2 \rightarrow L^2}
        \leq
        C_1 \exp \left( -t \nu^{\frac{\alpha}{\alpha+2}} F \left( 2^{2 \alpha+2} \Lambda(\alpha, 1, 2, u)^{2} - 2 \nu^{-2\frac{\alpha+1}{\alpha+2}}\frac{1}{n^2} \right) \right)
    \]
    for all $t>0$.

    At this point the proof goes on like that of \cite[Theorem 4.6]{GalGub2023}, allowing to obtain the same upper bound even for negative values of $\alpha$.
\end{proof}

We now have a sufficient condition for $u\in B_{1,\infty}^\alpha$ to be diffusion enhancing, namely that $\Lambda(\alpha+1, 1, 2, \partial_y U) > 0$, and we know that then the rate is also bounded from above in terms of $\nu$.
However, this would not be truly helpful unless we could say something on the subset of $B_{1,\infty}^\alpha$ for which positivity of $\Lambda$ holds.
This is where the notion of prevalence comes in: in fact, we are going to show that the set of $\alpha$-irregular elements is prevalent in the Besov space $B_{1,\infty}^\alpha$, even for negative $\alpha\in(-\frac12,0)$.


\section{About Wei's irregularity index}
\label{sec:Wei}

The purpose of this section is twofold.
On the one end, we analyse thoroughly Wei's irregularity index.
On the other, we aim at putting on solid ground the extension of Wei's condition to distributions.

In the first part we state a slightly more general version of Wei's condition, and give some alternative equivalent formulations.
The second part is devoted to some basic properties, connections with other definitions of irregularity, and a first result of stability.


\subsection{Wei's irregularity index}

For our goal, we first need to generalise Wei's irregularity condition; the original version \cite{Wei2021} corresponds, in the language of the following definition, to the case $k=1$, $p=2$.
We will then proceed by giving some equivalent formulations. Without loss of generality, we restrict to the domain $[0,1]$.

\begin{definition}[Wei's irregularity index]\label{d:lambda1}
    Given $p\in[1,\infty)$ and $f\in L^p(0,1)$, define for $\alpha\in\R$ and $k\in\N$,
    \[
        \Lambda(\alpha,k,p,f)
        = \inf_{\leb{J}\leq 1,P\in\Pc_k}\leb{J}^{-\alpha}\Bigl(\avint_J|f(x) - P(x)|^p\,dx\Bigr)^{\frac1p},
    \]
    where the infimum is extended over intervals $J\subset[0,1]$.
\end{definition}

For simplicity, we denote by $\|\cdot\|_{\Lnorm^p(J)}$ the $L^p$ norm over an interval $J$ with respect to the \emph{normalized} Lebesgue measure on $J$.

We first notice that in the case $p=2$ it is possible to explicitly identify the minimizing polynomials. To be more precise, define the polynomials $Q_0$, $Q_1$, \dots, such that $Q_0\equiv1$, $Q_i$ has degree $i$ and the family $(Q_i)_{i\geq0}$ is the orthonormal basis of $L^2(0,1)$ obtained by the Gram-Schmidt orthogonalization when starting with $Q_0\equiv1$. For any interval $J$ with leftmost point $y$ and length $\delta=\leb{J}$,
\[
    Q^J_i(x)
    = Q_i(\tfrac{x-y}\delta),
    \qquad x\in J,
\]
is an orthonormal basis of $\Lnorm^2(J)$.

\begin{proposition}\label{p:casetwo}
    Given $p\in[1,\infty]$ and $k\geq0$,
    \begin{equation}\label{eq:casetwo}
        \inf_{P\in\Pc_k}\|f-P\|_{\Lnorm^p(J)}
        \leq \|f-P_{k,J,f}\|_{\Lnorm^p(J)}
        \lesssim_{k,p} \inf_{P\in\Pc_k}\|f-P\|_{\Lnorm^p(J)}.
    \end{equation}
\end{proposition}
In other words one can replace the infimum over polynomials in \cref{d:lambda1} with the optimal (and explicit) polynomials of the case $p=2$ and obtain a comparable number, namely,
\[
    \Lambda(\alpha,k,p,f)
    \sim_{k,p} \inf_{\leb{J}\leq1}\leb{J}^{-\alpha}\|f-P_{k,J,f}\|_{\Lnorm^p(J)}.
\]
\begin{proof}[Proof of \cref{p:casetwo}]
    It is elementary to check that if $P\in\Pc_k$, then
    \[
        P(x)
        = \sum_{i=0}^k\langle P,Q^J_i\rangle Q^J_i(x),
        \qquad x\in J,
    \]
    and that if $f\in\Lnorm^2(J)$ and $k\geq0$,
    \begin{equation}\label{eq:poly2}
        P_{k,J,f}
        \vcentcolon= \sum_{i=0}^k\langle f,Q^J_i\rangle Q^J_i
        = \mathop{\text{argmin}}_{P\in\Pc_k}\|f-P\|_{\Lnorm^2(J)}.
    \end{equation}
    We preliminary notice that the inequality on the left-hand side of \eqref{eq:casetwo} is immediate, therefore it is sufficient to prove only the inequality on the right-hand side.
    We start with the case $k=0$. For an interval $J$, denote by $f_J$ the average of $f$ on $J$.
    By Jensen's inequality, for every $c\in\R$, $|c-f_J|\leq \|f-c\|_{\Lnorm^p(J)}$, therefore,
    \[
        \|f-f_J\|_{\Lnorm^p(J)}
        \leq \|f-c\|_{\Lnorm^p(J)} + |f_J-c|
        \leq 2\|f-c\|_{\Lnorm^p(J)},
    \]
    which concludes the proof in the case $k=0$.
    When $k\geq1$ the proof is conceptually similar. Indeed, if $J$ is an interval and $P\in\Pc_k$,
    \[
        |\langle P-f,Q^J_i\rangle|
        \leq \|f-P\|_{\Lnorm^p(J)}\|Q^J_i\|_{\Lnorm^{p'}(J)}
        \leq c(k,p)\|f-P\|_{\Lnorm^p(J)},
    \]
    where $\tfrac1p+\tfrac1{p'}=1$ and
    \[
        c(k,p')
        = \max_{i=0,1,\dots,k}\|Q^J_i\|_{\Lnorm^{p'}(J)}
        = \max_{i=0,1,\dots,k}\|Q_i\|_{L^{p'}(0,1)}.
    \]
    Therefore,
    \[
        \|P-P_{k,J,f}\|_{\Lnorm^p(J)}
        \leq c(k,p)\sum_{i=0}^k |\langle P-f,Q^J_i\rangle|
        \leq (k+1)c(k,p)c(k,p')\|f-P\|_{\Lnorm^p(J)},
    \]
    and
    \[
        \begin{aligned}
            \|f-P_{k,J,f}\|_{\Lnorm^p(J)}
             & \leq \|f-P\|_{\Lnorm^p(J)}
            + \|P-P_{k,J,f}\|_{\Lnorm^p(J)}                                \\
             & \leq \bigl(1+(k+1)c(k,p)c(k,p')\bigr)\|f-P\|_{\Lnorm^p(J)},
        \end{aligned}
    \]
    which concludes the proof.
\end{proof}

Denote by $\Dc$ the set of dyadic intervals of $[0,1]$, and by $\Dc_n\subset\Dc$ the dyadic intervals of length $2^{-n}$.

\begin{proposition}\label{p:dyadic}
    We have that
    \[
        \Lambda(\alpha,k,p,f)
        \sim_{\alpha,p} \inf_n 2^{n\alpha}\min_{J\in\Dc_n}\inf_{P\in\Pc_k}\|f-P\|_{\Lnorm^p(J)}.
    \]
\end{proposition}
\begin{proof}
    Given an interval $J\subset[0,1]$, if for an integer $n\geq0$, $2^{-(n+1)}<\leb{J}\leq 2^{-n}$, then there are $J_{n+2}\in\Dc_{n+2}$ and $J_{n-1}\in\Dc_{n-1}$ (unless $n=0$, in that case $J_{n-1}=[0,1]$) such that $J_{n-2}\subset J\subset J_{n-1}$.
    Therefore
    \[
        4^{-\frac1p}\|g\|_{\Lnorm^p(J_{n+2})}
        \leq \|g\|_{\Lnorm^p(J)}
        \leq 4^{\frac1p}\|g\|_{\Lnorm^p(J_{n-1})},
    \]
    and this is sufficient to conclude.
\end{proof}

We conclude this section with some auxiliary results that replace polynomials with discrete derivatives in the definition of Wei's index.

\begin{proposition}\label{p:def_delta}
    Let $\alpha>0$, $k\geq0$, $1\leq p<\infty$,
    \begin{align}
        \label{eq:delta_upper}
         & \inf_{J:\leb{J}\leq 1}\leb{J}^{-\alpha}\sup_{0\leq h\leq\leb{J}}\|\Delta_h^{k+1}f\|_{\Lnorm^p(J)}
        \lesssim_{\alpha,k,p}\Lambda(\alpha,k,p,f),                                                          \\
        \label{eq:delta_lower}
         & \Lambda(\alpha,0,p,f)
        \lesssim_{\alpha,k,p}\inf_{J:\leb{J}\leq 1}\leb{J}^{-\alpha}\bigl\|h\mapsto\|\Delta_h^1f\|_{\Lnorm^p(J)}\bigr\|_{\Lnorm^p(0,\leb{J})},
    \end{align}
    In particular
    \[
        \begin{aligned}
            \Lambda(\alpha,0,p,f)
             & \sim_{\alpha,k} \inf_{J:\leb{J}\leq 1}\leb{J}^{-\alpha}\bigl\|h\mapsto\|\Delta_h^1f\|_{\Lnorm^p(J)}\bigr\|_{\Lnorm^p(0,\leb{J})} \\
             & \sim_{\alpha,k} \inf_{J:\leb{J}\leq 1}\leb{J}^{-\alpha}\sup_{0\leq h\leq\leb{J}}\|\Delta_h^1f\|_{\Lnorm^p(J)}.
        \end{aligned}
    \]
\end{proposition}
\begin{proof}
    We start with the proof of the first inequality.
    Let $J$ be an interval, and $P\in\Pc_k$.
    Set $\delta=\leb{J}/(k+2)$ and let $J'$ be the interval having the same left endpoint of $J$ and length $\delta$.
    We notice that $\Delta_h^{k+1}P=0$, therefore, if $0\leq h\leq\delta$,
    \[
        \|\Delta^{k+1}_h f\|_{\Lnorm^p(J')}
        = \|\Delta^{k+1}_h (f - P)\|_{\Lnorm^p(J')}
        \lesssim_k \|f-P\|_{\Lnorm^p(J)}.
    \]
    Take the infimum over $P\in\Pc_k$ of the right-hand side, and the supremum over $h\in(0,\delta)$ on the left-hand side, so that
    \[
        \sup_{0\leq h\leq\delta}\|\Delta_h^{k+1}f\|_{\Lnorm^p(J')}
        \lesssim_k\inf_{P\in\Pc_k} \|f-P\|_{\Lnorm^p(J)}.
    \]
    Since $\leb{J'}=\leb{J}/(k+2)$, by bounding from below the left-hand side with its infimum over intervals, and then taking the infimum in $J$ of the left-hand side, the inequality follows.

    We turn to the proof of the second inequality.
    If $y\in(0,1)$ and $\delta\in(0,1]$, then
    \begin{equation}\label{eq:def_delta1}
        \Lambda(\alpha,0,p,f)^p
        \leq\delta^{-\alpha p}
        \avint_{(0,\delta)}|f(y+h)-f(y)|^p\,dh.
    \end{equation}
    Consider an interval $J$ with $\leb{J}=\delta$ and integrate the above inequality over $y\in J$ to get
    \[
        \Lambda(\alpha,0,p,f)
        \leq \leb{J}^{-\alpha}\bigl\|h\mapsto\|\Delta_h^1f\|_{\Lnorm^p(J)}\bigr\|_{\Lnorm^p(0,\leb{J})}.
    \]
    The conclusion follows by taking the infimum over $J$.
\end{proof}

\begin{corollary}\label{c:def_delta}
    If $\alpha>0$, $1\leq p<\infty$, and $f\in L^p(0,1)$, then
    \[
        \Lambda(\alpha,0,p,f)
        \lesssim_{\alpha,k} \inf_{\delta\in(0,1)}\delta^{-\alpha}
        \bigl\|h\mapsto\|\Delta_h^1f\|_{L^p(0,1)}\bigr\|_{\Lnorm^p(0,\delta)}.
    \]
\end{corollary}
\begin{proof}
    The inequality follows simply by integrating \eqref{eq:def_delta1} over all $(0,1)$.
\end{proof}

\begin{corollary}\label{l:deltainv}
    If $a>0$, $1\leq p<\infty$, then
    \[
        \Bigl(\avint_y^{y+\delta}|\Delta_\delta^{k+1}f(x)|^{-a}\,dx\Bigr)^{-\frac1a}
        \lesssim_{k,p}\inf_{P\in\Pc_k}\|f-P\|_{\Lnorm^p(y,y+(k+2)\delta)}.
    \]
\end{corollary}
\begin{proof}
    The inequality is elementary and follows as in \cite[Lemma 4.11]{GalGub2023}.
    By the Jensen inequality, since $x\mapsto x^{-a/p}$ is convex for $x>0$,
    \[
        \Bigl(\avint_y^{y+\delta}|\Delta_\delta^{k+1}f(x)|^p\,dx\Bigr)^{-\frac{a}{p}}
        \leq \avint_y^{y+\delta}|\Delta_\delta^{k+1}f(x)|^{-a}\,dx.
    \]
    The conclusion follows using \cref{p:def_delta}.
\end{proof}


\subsubsection{Basic properties}
Here we study some basic properties.
We first give some monotonicity properties of $\Lambda$ with respect to parameters.
In terms of irregularity, we show that if $\Lambda(\cdot,\cdot,\cdot,f)>0$ (for suitable values of the parameters), then $f$ is H\"older rough, and that if $f$ is $(\rho,\gamma)$-irregular, then $\Lambda(\cdot,\cdot,\cdot,f)>0$ (for suitable values of the parameters).
We finally state a stability result.

\begin{lemma}[Basic properties]\label{l:lambda_basic}
    Let $1\leq p<\infty$ and $f\in L^p(0,1)$.
    \begin{itemize}
        \item $\Lambda$ is monotone non-decreasing in $\alpha$ and $p$, and monotone non-increasing  in $k$, in the sense that
              \begin{itemize}
                  \item if $\alpha\leq\alpha'$, then $\Lambda(\alpha,k,p,f)\leq\Lambda(\alpha',k,p,f)$,
                  \item if $h\leq k$, then  $\Lambda(\alpha,k,p,f)\leq\Lambda(\alpha,h,p,f)$,
                  \item If $1\leq p\leq q<\infty$
                        and $f\in L^q(0,1)$, then $\Lambda(\alpha,k,p,f) \leq \Lambda(\alpha,k,q,f)$.
              \end{itemize}
        \item If $f_x\in L^1(0,1)$ and $k\geq1$,
              \[
                  \Lambda(\alpha,k,p,f)
                  \leq \Lambda(\alpha-1,k-1,1,f_x).
              \]
        \item If $\alpha\leq 0$,
              then $\Lambda(\alpha,k,p,f)=0$ for every $k\geq0$.
        \item If $\alpha>0$
              and $f^{(j)}\in L^1$, where $j=\lceil\alpha\rceil$, then $\Lambda(\alpha,k,p,f)=0$ for all $k\geq\alpha$.
    \end{itemize}
\end{lemma}
\begin{proof}
    The monotonicity properties are elementary.
    For the second property, let $P\in\Pc_{k-1}$ and let $J$ be an interval with leftmost point $y$.
    Let $Q\in\Pc_k$ be the primitive of $P$ such that $Q(y)=f(y)$, then
    for every $x\in J$,
    \[
        |f(x)-Q(x)|
        \leq \int_y^x|f_x(z) - P(z)|\,dz
        \leq \leb{J}\|f_x-P\|_{\Lnorm^1(J)},
    \]
    therefore
    \[
        \Lambda(\alpha,k,p,f)
        \leq \leb{J}^{-\alpha}\|f-Q\|_{\Lnorm^p(J)}
        \leq \leb{J}^{-(\alpha-1)}\|f_x-P\|_{\Lnorm^1(J)},
    \]
    The conclusion follows by taking the infimum over $P$ and $J$.

    For the third statement, assume that $\alpha\leq 0$. If $x\in(0,1)$ is a \emph{Lebesgue point} of $f$, we have that
    \[
        \Lambda(0,0,p,f)
        \leq \inf_{\delta>0}\Bigl(\avint_{(x-\delta,x+\delta)}|f(y)-f(x)|^p\,dy\Bigr)^{\frac1p}
        = 0.
    \]
    It then follows by the monotonicity properties of $\Lambda$ that $Lambda(\alpha,k,p,f)=0$ for all $\alpha\leq 0$ and $k\geq0$.

    Finally, the last statement is an elementary consequence of the second and the third.
\end{proof}


\subsubsection{Irregularity properties}

Here we prove that if a function is ``too'' regular, then its irregularity index is zero.
We then slightly extend some results of \cite{GalGub2023} about the connection of the irregularity index with other notions of irregularity.

\begin{lemma}\label{l:lambda_irregular}
    Let $\alpha>0$. If $1\leq p<\infty$ and $f$ is limit in $B^\alpha_{p,\infty}$ of smooth functions, then $\Lambda(\alpha,k,p,f)=0$ for all $k\geq[\alpha]$. In particular the conclusion holds if
    \begin{itemize}
        \item $1\leq p < \infty$ and $f\in B^{\alpha}_{p,q}$ with $q<\infty$,
        \item $1\leq p\leq\infty$ and $f\in B^{\alpha'}_{p,\infty}$ with $\alpha'>\alpha$.
    \end{itemize}
    Here $[\alpha]$ is the largest integer smaller than $\alpha$.
\end{lemma}
\begin{proof}
    Let $n=[\alpha]$, then by monotonicity, if $k\geq n$,
    \[
        \Lambda(\alpha,k,p,f)
        \leq \Lambda(\alpha-n,k-n,1,g)
        \leq \Lambda(\alpha-n,0,p,g),
    \]
    where $g$ is the $n^\text{th}$ derivative of $f$.
    Moreover, if $f\in B^\alpha_{p,q}$, then $g\in B^{\alpha-n}_{p,q}$ for any value of the indices.
    Therefore it is not restrictive to assume that $\alpha\in[0,1)$ and $k=0$.
    Moreover, the case $\alpha=0$ has been already examined in \cref{l:lambda_basic}, so we assume $\alpha\in(0,1)$.

    Let $f$ be limit in $B^\alpha_{p,\infty}$ of smooth functions.
    For $\epsilon>0$ let $\phi$ be a smooth function such that $\|f-\phi\|_{B^\alpha_{p,\infty}}\leq\epsilon$. By \cref{c:def_delta}, for every $\delta\in(0,1)$,
    \[
        \begin{aligned}
            \delta^\alpha\Lambda(\alpha,0,p,f)
             & \leq \bigl\|h\mapsto\|\Delta^1_h f\|_{L^p(0,1)}\bigr\|_{\Lnorm^p(0,\delta)}        \\
             & \leq \bigl\|h\mapsto\|\Delta^1_h (f-\phi)\|_{L^p(0,1)}\bigr\|_{\Lnorm^p(0,\delta)}
            + \bigl\|h\mapsto\|\Delta^1_h \phi\|_{L^p(0,1)}\bigr\|_{\Lnorm^p(0,\delta)}.
        \end{aligned}
    \]
    The first term can be bounded by the $B^\alpha_{p,\infty}$ norm,
    \[
        \bigl\|h\mapsto\|\Delta^1_h (f-\phi)\|_{L^p(0,1)}\bigr\|_{\Lnorm^p(0,\delta)}
        \leq\delta^\alpha\|f-\phi\|_{B^\alpha_{p,\infty}}
        \leq\delta^\alpha\epsilon,
    \]
    while the second can be bounded using the fact that $\phi$ is smooth,
    \[
        \bigl\|h\mapsto\|\Delta^1_h \phi\|_{L^p(0,1)}\bigr\|_{\Lnorm^p(0,\delta)}
        \leq \delta\|\phi\|_{W^{1,p}}.
    \]
    In conclusion
    \[
        \Lambda(\alpha,0,p,f)
        \lesssim \epsilon + \delta^{1-\alpha}\|\phi\|_{W^{1,p}}.
    \]
    By first taking the infimum in $\delta$ and then in $\epsilon$, the conclusion follows.

    We recall that if $p,q<\infty$, then smooth functions are dense in $B^\alpha_{p,q}$, and since $B^\alpha_{p,q}\subset B^\alpha_{p,\infty}$, all elements of $B^\alpha_{p,q}$ are limits of smooth functions in $B^\alpha_{p,\infty}$.
    This proves the second statement.

    Furthermore, if $\alpha'>\alpha$, then $B^{\alpha'}_{p,\infty}\subset B^\alpha_{p,q}$ for all $q<\infty$, so the third statement follows from the second, but for the case $p=\infty$, which can be proved easily by the following consideration: if $f\in B^{\alpha'}_{\infty,\infty}$, for every $y$ and every $\delta\in(0,1)$,
    \[
        \Lambda(\alpha,0,\infty,f)
        \leq \delta^{-\alpha}\sup_{x:|x-y|\leq\delta}|f(x)-f(y)|
        \leq\delta^{\alpha'-\alpha}\|f\|_{B^{\alpha'}_{\infty,\infty}},
    \]
    and the conclusion follows by taking the infimum in $\delta$.
\end{proof}

Here we extend to the full range of parameters some of the results of \cite{GalGub2023}.
Recall that \cite{FriHai2014}
\[
    L_\alpha(f)
    = \inf_{y,\delta}\sup_{x\in B_\delta(y)}\delta^{-\alpha}|f(x)-f(y)|.
\]
A measurable function $f$ is \emph{$\alpha$-H\"older rough} if $L_\alpha(f)>0$.

\begin{lemma}[Connection with H\"older roughness]\label{l:lambda_rough}
    Let $\alpha\in(0,1)$, then $\Lambda(\alpha,0,p,f)\leq 2^{\frac1p} L_\alpha(f)$.
    In particular, if $\Lambda(\alpha,k,p,f)>0$ for some $k\geq0$, then $f$ is $\alpha$-H\"older rough.
\end{lemma}
\begin{proof}
    For $y$, $\delta>0$, we have that
    \[
        \begin{multlined}[.9\linewidth]
            (\delta^\alpha\Lambda(\alpha,0,p,f))^p
            \leq \avint_y^{y+\delta}|f(x)-f(y)|^p\,dx\leq\\
            \leq 2\avint_{y-\delta}^{y+\delta}|f(x)-f(y)|^p\,dx
            \leq 2\delta^{\alpha p}\Bigl(\sup_{x\in B_\delta(y)}\delta^{-\alpha}|f(x)-f(y)|\Bigr)^p,
        \end{multlined}
    \]
    and by taking the infimum over $y$, $\delta$, the conclusion follows.
\end{proof}

Recall that, given $\rho>0$ and $\gamma\in[0,1)$, $f$ is \emph{$(\rho,\gamma)$-irregular} \cite{CatGub2016} if there is $c>0$ such that for every interval $J$,
\begin{equation}\label{eq:rhogamma}
    |\widehat\mu_f^J(\xi)|
    \leq c\leb{J}^\gamma(1+|\xi|)^{-\rho},
\end{equation}
where $\mu_J^f$ is the \emph{occupation measure} of $f$ on $J$,
\[
    \mu_f^J(A)
    = \int_J \mathbf{1}_A(f(x))\,dx.
\]
In the following proposition we show that $(\rho,\gamma)$-irregular functions satisfy a suitable version of Wei's irregularity condition, thus providing a partial answer to a question raised in \cite{GalGub2023}.

\begin{proposition}\label{p:rhogamma}
    Assume that $f\in L^p(0,1)$ is $(\rho,\gamma)$-irregular, for some $p\in[1,\infty)$, $\rho>0$ and $\gamma\in[0,1)$.
    Then
    \[
        \Lambda(\alpha,0,p,f)>0
    \]
    for $\alpha>\frac{1-\gamma}\rho$.
\end{proposition}
\begin{proof}
    Denote by $c_f$ the (best) constant appearing in \eqref{eq:rhogamma}, and assume first that $\rho\leq 1$.
    Fix an interval $J$, a number $a>0$ and denote by $A_a$ the following set,
    \[
        A_a=\{x\in J:|f(x)-f_J|\leq a\},
    \]
    where we recall that $f_J$ is the average of $f$ on $J$.
    We use an argument similar to \cite[Theorem 69]{GalGub2020}.
    Let $\psi:\R\to[0,1]$ be a bounded smooth even function, non-increasing on $[0,\infty)$ and $\psi\geq1$ on $[-1,1]$.
    Then, by Chebychev's inequality and Plancherel formula,
    \[
        \begin{multlined}[.9\linewidth]
            \leb{A_a}
            =\leb{\{x\in J:\psi((f(x)-f_J)/a)\geq1\}}
            \leq \int_J \psi\Bigl(\frac{f(x)-f_J}{a}\Bigr)\,dx =\\
            = \int_J \psi(u/a)\,\nu_J(du)
            = \int_J \overline{\psi(\cdot/a)}^{\wedge}(\xi)\widehat\nu_J(\xi)\,d\xi
            = a\int_J\widehat{\overline{\psi}}(a\xi)\widehat\nu_J(\xi)\,d\xi,
        \end{multlined}
    \]
    where $\nu_J$ is the occupation measure of $f-f_J$ on $J$.
    Notice that $\nu_J$ satisfies \eqref{eq:rhogamma} (with the same constant) if and only if $\mu_J^f$ does, therefore
    \[
        \begin{multlined}[.95\linewidth]
            \leb{A_a}
            \leq a c_f\leb{J}^\gamma\int_J (1+|\xi|)^{-\rho}|\widehat\psi(a\xi)|\,d\xi\leq\\
            \leq a c_f\leb{J}^\gamma
            \Bigl(\int_J (1+|\xi|)^{-q\rho}\,d\xi\Bigr)^{\frac1{q}}
            \Bigl(\int_J |\widehat\psi(a\xi)|^{q'}\,d\xi\Bigr)^{\frac1{q'}}
            \lesssim a^{\frac1{q}} c_f\leb{J}^\gamma,
        \end{multlined}
    \]
    where we have used \eqref{eq:rhogamma} and the H\"older inequality with exponents $q$, $q'$ such that $q\rho>1$.
    Choose now
    \[
        a
        =\epsilon c_f^{-q}\leb{J}^{q(1-\gamma)},
    \]
    with $\epsilon>0$ small enough that
    \[
        \frac1{\leb{J}}\leb{A_a}
        \lesssim a^{\frac1{q}} c_f\leb{J}^{\gamma-1}
        \lesssim \epsilon^{\frac1q}
        \ll 1.
    \]
    In conclusion, we have proved that a $(\rho,\gamma)$-irregular function has small oscillations in a small part of any interval: for $q\rho>1$ and $\epsilon>0$,
    \begin{equation}\label{eq:smallsmall}
        \frac1{\leb{J}}\leb{\{x\in J:|f(x)-f_J|\leq \epsilon c_f^{-q}\leb{J}^{q(1-\gamma)}\}}
        \lesssim\epsilon^{\frac1q},
    \end{equation}
    for every interval $J$.

    Let us prove now that if \eqref{eq:smallsmall} holds, then $\Lambda(\alpha,0,p,f)>0$ for $\alpha=q(1-\gamma)$ and all $p$ such that $f\in L^p$.
    For an interval $J$, consider the set $A_a$ with the above choice of the parameter $a$, and $\epsilon$ small enough, then
    \[
        \begin{multlined}[.9\linewidth]
            \|f-f_J\|_{\Lnorm^p(J)}
            \geq \|(f-f_J)\mathbf{1}_{A_a^c}\|_{\Lnorm^p(J)}\gtrsim\\
            \gtrsim c_f^{-q}\leb{J}^{q(1-\gamma)}\Bigl(1 - \frac{\leb{A_a}}{\leb{J}}\Bigr)^{\frac1p}
            \gtrsim c_f^{-q}\leb{J}^{q(1-\gamma)},
        \end{multlined}
    \]
    so that in conclusion by \cref{p:casetwo},
    \[
        \Lambda(q(1-\gamma),0,p,f)
        \gtrsim c_f^{-q},
    \]
    which is positive as long as $c_f<\infty$.
    Since $\rho\leq1$, $q$ can be chosen arbitrarily close to $1/\rho$, and thus $q(1-\gamma)$ is arbitrarily close to $(1-\gamma)/\rho$.
    If on the other hand $\rho>1$, we notice \cite[Lemma 7]{GalGub2020} that $f$ is $(\rho_\theta,\gamma_\theta)$-irregular for $\rho_\theta=\theta\rho$ and $\gamma_\theta=1-\theta+\theta\gamma$, and $(1-\gamma_\theta)/\rho_\theta=(1-\gamma)/\rho$.
    So we can choose $\theta$ so that $\rho_\theta\leq1$, use the arguments already developed and conclude, since $q(1-\gamma_\theta)$ can be chosen arbitrarily close to $(1-\gamma_\theta)/\rho_\theta=(1-\gamma)/\rho$.
\end{proof}

We have not found a way to prove any of the opposite implications of \cref{p:rhogamma} above, namely neither that \emph{small oscillations are a small set} (formula \eqref{eq:smallsmall}) yields $(\rho,\gamma)$-irregularity, nor that the Wei condition of order $k=0$ yields \eqref{eq:smallsmall}.


\subsubsection{Stability with respect to regular perturbations}

We recall the definition of Campanato spaces $\Lc^{p,\alpha}$, with $1\leq p<\infty$ and $\alpha>0$ (see for instance \cite{RafSamSam2013}) defined through the seminorm (which we have adapted for easier comparison with the generalized Wei's condition)
\begin{equation}\label{eq:campanato0}
    [f]_{\Lc^{p,\alpha}}
    \vcentcolon= \sup_{\delta>0}\delta^{-\alpha}\bigl(\sup_{\leb{J}=\delta}\|f-f_J\|_{L^p(J)}\bigr)
    \sim\sup_{\delta>0}\delta^{-\alpha}\bigl(\sup_{\leb{J}=\delta}\inf_{c\in\R}\|f-c\|_{L^p(J)}\bigr),
\end{equation}
where the supremum is extended over intervals $J$.
Campanato spaces
\begin{itemize}
    \item are equal to Morrey spaces $L^{p,\alpha}$ (the definition is similar but without averages) when $\alpha\in[0,1/p)$,
    \item are equal to BMO when $\alpha=1/p$,
    \item are equal to H\"older spaces $C^{\alpha-1/p}$ when $\alpha\in(1/p,1+1/p]$,
    \item contain only constant functions when $\alpha>1+1/p$.
\end{itemize}
One can define also higher order Campanato spaces $\Lc^{p,\alpha}_k$ through polynomials \cite{Cam1964},
\begin{equation}\label{eq:campanatok}
    [f]_{\Lc^{p,\alpha}}
    \vcentcolon= \sup_{r>0}r^{-\alpha}\bigl(\sup_{\leb{J}\leq1}\inf_{P\in\Pc_k}\|f-P\|_{L^p(J)}\bigr),
\end{equation}
and
\begin{itemize}
    \item if $\alpha\in[0,1/p)$, $\Lc^{p,\alpha}_k$ is the Morrey space $L^{p,\alpha}$,
    \item if $m+1/p<\alpha<(m+1)+1/p$ and $0\leq m\leq k$, $\Lc^{p,\alpha}_k=C^{m,\alpha-1/p-m}$.
\end{itemize}

\begin{lemma}[Stability]\label{l:stability}
    If $f\in L^p(0,1)$ and $\varphi\in \Lc_k^{p,\alpha+1/p}$, then
    \[
        \Lambda(\alpha,k,p,f+\varphi)
        \lesssim_{k,p} \Lambda(\alpha,k,p,f) + [\varphi]_{\Lc^{p,\alpha+1/p}_k}.
    \]
    In particular, if $\Lambda(\alpha,k,p,f)>0$ and $[\varphi]_{\Lc^{p,\alpha+1/p}_k}$ is sufficiently small, then $\Lambda(\alpha,k,p,f+\varphi)>0$.
\end{lemma}
\begin{proof}
    By \cref{p:casetwo} and by the fact that the polynomials in \eqref{eq:poly2} are linear in $f$, it follows that
    \[
        \inf_{P\in\Pc_k}\|f+\varphi - P\|_{\Lnorm^p(J)}
        \lesssim_{k,p}\inf_{P\in\Pc_k}\|f-P\|_{\Lnorm^p(J)}
        + \inf_{P\in\Pc_k}\|\varphi-P\|_{\Lnorm^p(J)},
    \]
    therefore
    \[
        \Lambda(\alpha,k,p,f+\varphi)
        \lesssim_{k,p} \leb{J}^{-\alpha}\inf_{P\in\Pc_k}\|f-P\|_{\Lnorm^p(J)}
        + [\varphi]_{\Lc^{p,\alpha+1/p}_k}.
    \]
    By taking the infimum over all intervals $J$, with $\leb{J}\leq 1$, the conclusion follows.

    For the second statement, notice that
    \[
        \Lambda(\alpha,k,p,f)
        = \Lambda(\alpha,k,p,(f+\varphi)-\varphi)
        \lesssim_{k,p}\Lambda(\alpha,k,p,f+\varphi)
        + [\varphi]_{\Lc^{p,\alpha+1/p}_k},
    \]
    so if $\Lambda(\alpha,k,p,f)>0$ and $[\varphi]_{\Lc^{p,\alpha+1/p}_k}$ is sufficiently small, then $\Lambda(\alpha,k,p,f+\varphi)>0$ as well.
\end{proof}

One would actually expect that $\Lambda(\alpha,k,p,f)>0$ if and only if $\Lambda(\alpha,p,k,f+\varphi)>0$, indipendently from the magnitude of the norm of $\varphi$ (in some suitable space of regular functions).
To this end we define the following variant of the irregularity index, that supports the idea that irregularity is a matter of small scales.

\begin{equation}\label{eq:ourwei}
    \Lambda_\ell(\alpha,k,p,f)
    = \liminf_{\delta\to0}\inf_{\leb{J}\leq\delta,P\in\Pc_k}\leb{J}^{-\alpha}\|f(x) - P(x)\|_{\Lnorm^p(J)},
\end{equation}

We first prove that the local index captures irregularity as much as the original index.

\begin{theorem}\label{t:ourwei}
    Let $\alpha>0$, $k\geq0$, $1\leq p<\infty$. Given $f\in L^p(0,1)$,
    \[
        \Lambda(\alpha,k,p,f) = 0
        \quad\iff\quad\Lambda_\ell(\alpha,k,p,f) = 0.
    \]
\end{theorem}
\begin{proof}
    Using the arguments of \cref{p:dyadic} and \cref{p:casetwo}, it follows that
    \begin{equation}\label{eq:liminf_dyadic}
        \Lambda_\ell(\alpha,k,p,f)
        \sim_{\alpha,k,p}\liminf_n 2^{n\alpha}\min_{J\in\Dc_n}\|f - P_{k,J,f}\|_{\Lnorm^p(J)},
    \end{equation}
    and, likewise,
    \begin{equation}\label{eq:inf_dyadic}
        \Lambda(\alpha,k,p,f)
        \sim_{\alpha,k,p}\inf_n 2^{n\alpha}\min_{J\in\Dc_n}\|f - P_{k,J,f}\|_{\Lnorm^p(J)}.
    \end{equation}
    Since $\Lambda(\alpha,k,p,f)\leq\Lambda_\ell(\alpha,k,p,f)$, it remains only to show that $\Lambda_\ell(\alpha,k,p,f)=0$ when $\Lambda(\alpha,k,p,f)=0$.

    Assume $\Lambda(\alpha,k,p,f)=0$ and, by contradiction, that $\Lambda_\ell(\alpha,k,p,f) \geq \epsilon_0 > 0$.
    By \eqref{eq:liminf_dyadic}, for $n$ large enough,
    \[
        2^{n\alpha}\min_{J\in\Dc_n}\|f - P_{k,J,f}\|_{\Lnorm^p(J)}
        \geq\frac12\epsilon_0,
    \]
    therefore the infimum in \eqref{eq:inf_dyadic} is attained  at some finite value of $n$.
    In conclusion there are $n\geq0$ and $J\in\Dc_n$ such that $\|f - P_{k,J,f}\|_{L^p(J)}=0$, that is $f$ is {a.\,e.} equal to a polynomial on $J$, and so is regular and $\Lambda_\ell(\alpha,k,p,f)=0$.
\end{proof}

It turns out that stability by (suitable) regular perturbations is more robust in terms of the index $\Lambda_\ell$.
Indeed, recall the vanishing Campanato space $\Vc^{p,\alpha}_k$, defined by the condition
\begin{equation}\label{eq:vcs}
    \lim_{\delta\to0}\sup_{\leb{J}\leq\delta}\leb{J}^{-\alpha}\inf_{P\in\Pc_k}\|f-P\|_{L^p(J)}
    = 0.
\end{equation}
This is a closed subspace of $\Lc^{p,\alpha}_k$ that contains the closure of smooth functions in $\Lc^{p,\alpha}_k$.
We conjecture that actually $\Vc^{p,\alpha}_k$ is equal to the closure of smooth functions, although the proof of this statement goes beyond the scopes of this paper.
This statement is proved in the case $\alpha\in(0,1)$ in \cite{Ope2003}.

\begin{proposition}\label{p:stability_revisited}
    If $f\in L^p(0,1)$ and $\varphi\in \Vc_k^{p,\alpha+1/p}$, then
    \[
        \Lambda_\ell(\alpha,k,p,f+\varphi)
        \sim_{k,p} \Lambda_\ell(\alpha,k,p,f).
    \]
    In particular, $\Lambda(\alpha,k,p,f+\varphi)>0$ if and only if $\Lambda(\alpha,k,p,f)>0$.
\end{proposition}
\begin{proof}
    Fix $\delta>0$ and let $J$ be an interval with $\leb{J}=\delta$.
    We use \cref{p:casetwo} to obtain that
    \[
        \inf_{P\in\Pc_k}\|f+\varphi-P\|_{\Lnorm^p(J)}
        \lesssim_{k,p} \inf_{P\in\Pc_k}\|f-P\|_{\Lnorm^p(J)} + \inf_{P\in\Pc_k}\|\varphi-P\|_{\Lnorm^p(J)}.
    \]
    Denote by $H_\varphi(\delta)$ the value, with $f$ replaced by $\varphi$, whose limit is $0$ in \eqref{eq:vcs}.
    Then
    \[
        \inf_{P\in\Pc_k}\|\varphi-P\|_{\Lnorm^p(J)}
        \leq \delta^\alpha H_\varphi(\delta),
    \]
    therefore,
    \[
        \Lambda_\ell(\alpha,k,p,f+\varphi)
        \lesssim\delta^{-\alpha}\inf_{P\in\Pc_k}\|f-P\|_{\Lnorm^p(J)} + H_\varphi(\delta).
    \]
    By taking the infimum over intervals (notice that just one quantity depends on the intervals), and then by taking the limit $\delta\to0$, it follows that
    \[
        \Lambda_\ell(\alpha,k,p,f+\varphi)
        \lesssim_{k,p} \Lambda_\ell(\alpha,k,p,f).
    \]
    By applying the same inequality to $f=(f+\varphi)-\varphi$, the opposite estimate follows.
\end{proof}


\section{Prevalence of diffusion enhancing velocities}
\label{sec:prevalence}

Building on \cref{sec:Wei} and on the prevalence results developed in \cite{GalGub2020,GalGub2023}, to whom we refer for an introduction of such notion, we prove here that diffusion enhancing velocities form a prevalent set in $B^\alpha_{1,\infty}$, ensuring that the bounds in \cref{sec:enhdiff} are not purely theoretical.

First of all set, for $\alpha>0$ and $\lambda>0$,
\[
    G_\alpha(y,\delta,k,f)
    \coloneqq
    \avint_y^{y+\delta}|\Delta_\delta^{k+1}f(x)|^{-\frac1\alpha}\,dx
\]
and
\[
    K(\alpha,\lambda,k,f)
    \coloneqq
    \sup_{n\geq0,m=1,2,\dots,2^n-1} 2^{-\lambda n}G_\alpha(y_{n,m},\delta_n,k,f),
\]
where $\delta_n \coloneqq \frac\pi{2^{n+1}}$ and $y_{n,m} \coloneqq -\pi + 2\pi \frac{m}{2^n}$.
These quantities will prove useful in showing that Wei's condition holds for almost all elements of $B^\alpha_{1,\infty}$ thanks to \cref{lemma:WeigreaterthanK} below, which allows us to work with them instead of $\Lambda$.

\begin{lemma}\label{lemma:WeigreaterthanK}
    For $\lambda>1$, $\beta>0$ and $1 \leq p < \infty$,
    \[
        \Lambda(\lambda\beta,p,k,f)
        \gtrsim_{\alpha,k,p} K(\alpha,\lambda,k,f)^{-\alpha}
    \]
\end{lemma}
\begin{proof}
    The inequality follows as in \cite[Lemma 4.12]{GalGub2023}.
    Fix $y$ and $\delta$ and let $n,m$ be such that $\delta\in(\pi2^{-n},\pi2^{-n+1}]$ and $y\in[y_{n,m-1},y_{n,m}]$, so that $[y_{n,m},y_{n,m}+(k+2)\delta_n]\subset[y,y+(k+2)\delta]$.
    So, by \cref{l:deltainv},
    \[
        \begin{multlined}[.95\linewidth]
            \delta^{-\lambda\alpha}\inf_{P\in\Pc_k}\|f-P\|_{\Lnorm^p(y,y+(k+2)\delta)}
            \gtrsim_{k,p}
            \delta^{-\lambda\alpha}\Bigl(\avint_{y_{m,n}}^{y_{m,n}+\delta_n}|\Delta_\delta^{k+1}f(x)|^{-\frac1\alpha}\,dx\Bigr)^{-\alpha} \\
            =
            \delta^{-\lambda\alpha}G_\alpha(y_{n,m},\delta_n,k,f)^{-\alpha}
            \gtrsim
            K(\alpha,\lambda,k,f)^{-\alpha} ,
        \end{multlined}
    \]
    the last inequality following from the definition of $K$.
\end{proof}

The rest of this section is divided in two parts: \cref{subsec:prevforstoch} is devoted devoted to the proof of \cref{th:prevgeneral}, which ensures that elements of Besov spaces of suitable regularity satisfy Wei's condition, while \cref{subsec:abstractresult} connects such result to diffusion enhancement via \cref{th:final}.


\subsection{Stochastic processes and prevalence}\label{subsec:prevforstoch}

We first recall the notion of \emph{strongly local nondeterminism}, see \cite{MonPit1987}, and \cite{Xia2007} for a more recent review on the subject.
A Gaussian process $(X_t)_{t\geq0}$is $H$-strongly locally nondeterministic ($H$-SLND) if there is $c>0$ such that for every $0\leq s\leq t$,
\[
    \Var(X_t\mid\mathcal{F}_s)
    \geq c|t-s|^{2H},
\]
where $\mathcal{F}_s=\sigma(X_r:0\leq r\leq s)$ is the fuiltration generated by the process.

\begin{lemma}\label{lemma:finexpmomcond}
    Let $X$ be a stochastic process, $X:[0,\pi]\rightarrow\R$.
    If, for some $\lambda>0$ and $\alpha\in\R$,
    \[
        \sup_{\delta\in(0,1),y\in[0,\pi-\delta]}\E\Bigl[e^{\lambda G_\alpha(y,\delta,k,X)}\Bigr]
        <\infty,
    \]
    then $\Pb\left(\Lambda(\beta,p,k,X)>0\right)=1$ for all $\beta>\alpha$.
\end{lemma}
\begin{proof}
    By \cref{lemma:WeigreaterthanK}, given $\lambda>1$,
    \[
        \Pb\left( \Lambda(\lambda\alpha,p,k,X)>0 \right)
        \geq
        \Pb\left( K(\alpha,\lambda,k,X) < \infty \right) .
    \]
    Set now
    \[
        Y
        \coloneqq
        \sum_{n\in\N} 2^{-2n} \sum_{m=1}^{2^n-1} \exp\left( \lambda G_\alpha(y_{m,n}, \delta_n, k, X) \right)
    \]
    and notice that $\E \left[ Y \right] < \infty$, hence $\Pb \left(Y<\infty \right) = 1$.
    Since for all $m,n$ we have $G_\alpha(y_{m,n}, \delta_n, k, X) \leq \frac{1}{\lambda} \log(2^{2n} Y) \lesssim \frac{n}{\lambda} (1+\log Y)$, the thesis follows because
    \[
        \sup_{m,n\in\N} \frac1n G_\alpha(y_{m,n}, \delta_n, k, X) \lesssim \frac1\lambda (1+\log Y) < \infty \quad \Pb\text{-a.s.}
    \]
    and $2^{-n\lambda} \lesssim_\lambda n^{-1}$ for all $\lambda > 0$.
\end{proof}

\begin{lemma}\label{lemma:finexpmom}
    Let $X:[0,\pi]\rightarrow\R$ be a Gaussian and $H$-strongly locally nondeterministic stochastic process, with $H>1$.
    Then for all $\alpha>H$ and $k\geq1$ there exists $\lambda>0$ such that
    \[
        \sup_{\delta\in(0,1),\bar{y}\in[0,\pi-\delta]}\E\Bigl[e^{\lambda G_\alpha(\bar{y},\delta,k,Y)}\Bigr]
        <\infty.
    \]
\end{lemma}
\begin{proof}
    Consider the process $Y = |\Delta_\delta^{k+1} X|^{-\frac1\alpha}$ and the filtrations $\mathcal{F}_y = \sigma\left( \left\{ X_z, z\leq y \right\} \right)$, $\mathcal{G}_y = \mathcal{F}_{y+(k+2)\delta}$.
    Clearly, $\Delta_\delta^{k+1} X(y)$ is $\mathcal{G}_y$ adapted and $\forall [x,y] \subseteq [\bar{y}, \bar{y}+(k-1)\delta]$ we have
    \[
        \Var\left( \Delta_\delta^{k+1} X(y) | \mathcal{G}_x \right)
        =
        \Var\left( X(y+(k+1)\delta) | \mathcal{F}_{x+(k+1)\delta} \right)
        \geq C_X |y-x|^{2H} .
    \]
    This means that $\Delta_\delta^{k+1} X$ is itself $H$-SLND, hence we can write $\Delta_\delta^{k+1} X(y) = Z_{x,y} + \tilde{Z}_{x,y}$, with $Z_{x,y}$ adapted to $\mathcal{G}_x$ and $\tilde{Z}_{x,y}$ Gaussian and independent from $\mathcal{G}_x$.
    It follows that
    \[
        \E \left[ \int_z^{\bar{y}+\delta} |\Delta_\delta^{k+1} X(y) |^{-{\frac{1}{\alpha}}} \bigg| \mathcal{G}_z \right]
        =
        \int_z^{\bar{y}+\delta} \E \left[ |Z_{z,y} + \cdot |^{-{\frac{1}{\alpha}}} \right] (\tilde{Z}_{z,y}) dy
    \]
    and we can therefore proceed like \cite[Lemma 4.16]{GalGub2023}), eventually obtaining that
    \[
        \E \left[ \int_z^{\bar{y}+\delta} |\Delta_\delta^{k+1} X(y) |^{-{\frac{1}{\alpha}}} \bigg| \mathcal{G}_z \right]
        \lesssim
        \int_z^{\bar{y}+\delta} |y-z|^{-\frac H\alpha} d y
        \leq \int_0^1 |y|^{-\frac H\alpha} \coloneqq C .
    \]
    Since $C$ is deterministic and uniform over the choices of $\delta\in(0,1)$, $\bar{y}\in\Tb$ and $z\in[\bar{y},\bar{y}+\delta]$, \cite[Lemma 4.14]{GalGub2023} yields the thesis.
\end{proof}

\begin{theorem}\label{th:SLND->Wei}
    If $X:[0,\pi]\rightarrow\R$ is a Gaussian and $H$-SLND stochastic process, with $H>1$, then
    \[
        \Pb \left( \Lambda(\alpha,p,k,X)>0 \right) =1,
    \]
    for all $\alpha>H$, $1\leq p<\infty$ and $k\geq1$.
\end{theorem}
\begin{proof}
    The proof is a simple combination of \cref{lemma:finexpmomcond,lemma:finexpmom}.
\end{proof}

\begin{theorem}\label{th:prevgeneral}
    If $1\leq p<\infty$, $k\geq0$ and $\alpha>0$, then the set
    \[
        \{f\in B^\alpha_{p,\infty}(0,\pi): \Lambda(\beta,p,k,f)>0\text{ for all }\beta>\alpha\}
    \]
    is prevalent in $B^\alpha_{p,\infty}(0,\pi)$.
\end{theorem}
\begin{proof}
    First of all,
    \begin{multline*}
        \{f\in B^\alpha_{p,\infty}: \Lambda(\beta,p,k,f)>0\text{ for all }\beta>\alpha\} \\
        =
        \bigcap_{m = 1}^\infty \bigcup_{n=1}^\infty \mathcal{B}_{m,n}
        \coloneqq
        \bigcap_{m = 1}^\infty \bigcup_{n=1}^\infty \left\{ f\in B_{p,\infty}^\alpha : \Lambda(\alpha + \frac1m, p, k, f) \geq \frac1n \right\}
    \end{multline*}
    is Borel measurable, since all $\mathcal{B}_{m,n}$ are closed in $B_{p,\infty}^\alpha$ due to $f\mapsto\Lambda(\beta,p,k,f)$ being upper semicontinuous in the $L^p(0,\pi)$ topology.

    We now just have to show that $\bigcup_{n=1}^\infty \mathcal{B}_{m,n}$ is prevalent for all $n$, since countable intersection of prevalent sets is itself prevalent.
    To this aim, consider some $\beta>\alpha$ and fix $H\in(\alpha,\beta)$.
    Let then $\{f_y\}_{y\in[0,\pi]}$ be the canonical fractional Brownian motion of Hurst parameter $H$ on $C(0,\pi)$ and $\mu^H$ be its law, which is tight on $B_{p,\infty}^\alpha(0,\pi)$ (cf. e.g. \cite[Section 2.1]{GalGub2023}).

    Because $\{f_y\}_{y\in[0,\pi]}$ is $H$-SLND, any primitive $\int f$ of it is $(H+1)$-SLND and so is any process of the form $\int f +\int \phi$, $\phi\in B_{p,\infty}^\alpha$.
    By means of \cref{th:SLND->Wei}, it follows that
    \begin{multline*}
        \mu^H (\phi + \{ f\in B_{p,\infty}^\alpha : \Lambda(\beta+1, p, h, \int f) \geq 0 \}) \\
        =
        \mu^H (\{ f\in B_{p,\infty}^\alpha : \Lambda(\beta+1, p, h, \int f + \int \phi) \geq 0 \}) = 1 ,
    \end{multline*}
    for any $\phi\in B_{p,\infty}^\alpha$, hence the set $\left\{ f\in B_{p,\infty}^\alpha : \Lambda(\beta+1, p, h, \int f) \geq 0 \right\}$ is prevalent in $B_{p,\infty}^\alpha$ for $h\geq1$.
    We can now conclude noting that, by \cref{l:lambda_basic},
    \[
        \left\{ f\in B_{p,\infty}^\alpha : \Lambda(\beta+1, p, h, \int f) \geq 0 \right\}
        \subseteq
        \left\{ f\in B_{p,\infty}^\alpha : \Lambda(\beta, p, h-1, f) \geq 0 \right\}
    \]
    hence the latter is also prevalent for all $\beta>\alpha$ and $h\geq 1$.
\end{proof}

It is important to notice here that the results leading to \cref{th:prevgeneral}, as well as those in \cref{sec:Wei}, hold for elements with a finite domain ($[0,\pi]$, without loss of generality).
In order to retrieve them on the whole Torus, we must resort to a solution of the type \cite[Corollary 4.19]{GalGub2023}, which extends the proof to $\Tb$.
In other words, instead of considering $\mu^H$, the law of a fractional Brownian motion on $[0,\pi]$, as the measure witnessing the prevalence of $\bigcup_{n=1}^\infty \mathcal{B}_{m,n}$, we take $\tilde\mu^H = T_\#\mu$, where $(Tf)(y) = f(|y|)$.
$\mu^H$ is tight on $B_{p,\infty}^\alpha(0,\pi)\cap L^\infty(0,\pi)$, hence $\tilde\mu^H$ will be tight on $B_{p,\infty}^\alpha(\Tb)$.
This, combined with the fact that $T$ lets any $f:\Tb\to\R$ split into $f_1,f_2:[0,\pi]\to\R$, accounting respectively for negative and positive values, allows to extend \cref{th:prevgeneral} to the whole Torus.


\subsection{Abstract result and conclusion}\label{subsec:abstractresult}

In this subsection we will again be working work on $\Tb$, avoiding to specify it every time for the sake of notation.

\begin{proposition}\label{prop:abstract}
    If $E$, $F$ are complete metric vector spaces, if $T:E\to F$ is linear, continuous, and one-to-one, and if $A$ is prevalent in $E$, then $T(A)$ is prevalent in $F$.
\end{proposition}
\begin{proof}
    Let $S\subset E$ be shy: there are a compact set $K\subset E$ and a measure $\mu$ on $E$ such that
    $\mu(K)\in(0,\infty)$
    and
    $\mu(x+S)=0$ for all $x\in E$.
    Let us prove that $T(S)$ is shy in $F$.

    Consider the measure $\nu=T_\#\mu$ on $F$, and the set $K'=T(K)$.
    \begin{enumerate}[label=\roman*)]
        \item Since $T$ is continuous, it follows that $K'$ is compact.
        \item Since $T$ is injective, $T^{-1}T(K)=K$, therefore $\nu(K')=\mu(K)\in(0,\infty)$.
        \item Since $T$ is injective, surjective and linear, $T^{-1}(y+T(S))=T^{-1}y + T^{-1}T(S)=T^{-1}y + S$, therefore for every $y\in F$,
              \[
                  \nu(y+T(S))
                  = \mu(T^{-1}(y+T(S)))
                  = \mu(T^{-1}y + S)
                  = 0.
              \]
    \end{enumerate}
    Finally, since $T$ is one-to-one, $A^c$ is shy, $T(A)^c=T(A^c)$ is shy, and thus $T(A)$ is prevalent.
\end{proof}

\begin{corollary}
    Let $\alpha\in\R$, $1\leq p<\infty$, $k\in\N$ such that $\alpha+k>0$, and let $T:B^\alpha_{p,\infty}\to B^{\alpha+k}_{p,\infty}$ linear, continuous and one-to-one.
    Then
    \[
        \{f\in B^\alpha_{p,\infty}:\ \Lambda(\beta+k,p,k,Tf)>0\text{ for all }\beta>\alpha\}
    \]
    is prevalent in $B^\alpha_{p,\infty}$.
\end{corollary}
\begin{proof}
    If $\beta>\alpha$,
    \[
        T(\{f\in B^\alpha_{p,\infty}:\ \Lambda(\beta+k,p,k,Tf)>0\})
        = \{F\in B^{\alpha+k}_{p,\infty}:\ \Lambda(\beta+k,p,k,F)>0\}
    \]
    and the set on the right-hand side is prevalent in $B^{\alpha+k}_{p,\infty}$.
    The conclusion follows by the previous proposition applied to the map $T^{-1}$.
\end{proof}

\begin{theorem}\label{th:final}
    Almost every velocity in $B^\alpha_{1,\infty}$, $\alpha\in(-\frac12, 0)$ is diffusion enhancing.
\end{theorem}
\begin{proof}
    In order to prove the theorem, we show that almost every element of $B^\alpha_{1,\infty}$ satisfy the sufficient condition \eqref{eq:ourWei}.
    To this aim, notice that for $\gamma>0$ it is not hard to prove a version of \cref{th:prevgeneral} on the space $\mathring{B}_{1,\infty}^\gamma$ (which denotes, with a slight abuse of notation, the zero-mean subset of $B_{1,\infty}^\gamma$); in particular, given $\alpha\in(-\frac12, 0)$, the set $\{f\in\mathring{B}_{1,\infty}^{\alpha+2}: \Lambda(\beta+2, 2, 2, f)>0\}$ is prevalent in $\mathring{B}_{1,\infty}^{\alpha+2}$ for all $\beta>\alpha$.

    Moreover, since $\partial_y: \mathring{B}_{1,\infty}^{\alpha+2} \rightarrow \mathring{B}_{1,\infty}^{\alpha+1}$ satisfies the hypothesis of \cref{prop:abstract}, the set
    \[
        \{f\in\mathring{B}_{1,\infty}^{\alpha+1}: \Lambda(\beta+1, 1, 2, f)>0\} \supseteq \{f\in\mathring{B}_{1,\infty}^{\alpha+2}: \Lambda(\beta+2, 2, 2, \partial_y^{-1}f)>0\}
    \]
    is prevalent in $\mathring{B}_{1,\infty}^{\alpha+1}$. The inclusion in guaranteed by \cref{l:lambda_basic}.

    We conclude observing that the solution operator $S$ of the well-posed problem \eqref{eq:itotrickII} also satisfies the hypothesis of \cref{prop:abstract}.
    Hence, given the operator $T:B_{1,\infty}^\alpha\rightarrow\mathring{B}_{1,\infty}^{\alpha+1}$, defined as $Tu = \partial_y (Su) = \partial_y U$ with $U$ solution of \eqref{eq:itotrickII}, we have that the set
    \[
        T^{-1} \{f\in\mathring{B}_{1,\infty}^{\alpha+1}: \Lambda(\alpha+1, 1, 2, f)>0\} = \{u\in\mathring{B}_{1,\infty}^{\alpha}: \Lambda(\alpha+1, 1, 2, \partial_y U) >0\}
    \]
    is prevalent in $B_{1,\infty}^\alpha$.
\end{proof}


\appendix

\section{Well-posedness in the negative regularity regime}
\label{app:wellposedness}

We prove here that equation \cref{eq:hypoelliptic} is well posed whenever the velocity term $u$ belongs to a Besov space $B_{1,\infty}^\alpha$, provided that $\alpha\in(-\frac{1}{2},0)$.

\begin{theorem}\label{th:wp}
    Let $f_0\in L^2(\Tb)$ and $u\in B_{1,\infty}^\alpha$, $-1/2 < \alpha < 0$.
    Then the equation
    \begin{equation}\label{eq:problem}
        \partial_t f + \imath u f = \nu \partial^2_y f
    \end{equation}
    admits a unique weak solution.
    Moreover, such solution belongs to $L^2(0,T;H^1(\Tb)) \cap L^\infty(0,T;L^2(\Tb))$.
\end{theorem}
\begin{proof}
    Assume first that that $u$ is smooth.
    Then it makes sense to multiply the equation by $f$ and integrate over $\Tb$ to get
    \[
        \frac{1}{2}\partial_t \norm{f}{L^2}^2 + \nu \norm{\partial_y f}{L^2}^2 = 0 ,
    \]
    which yields the regularity estimates.
    Indeed, it suffices to integrate in time to observe that
    \[
        \sup_{t\in[0,T]} \norm{f_t}{L^2}^2 \leq \norm{f_0}{L^2}^2
        \quad \text{and} \quad
        \int_0^T \norm{\partial_y f_t}{L^2}^2 dt \leq \frac{1}{2\nu}\norm{f_0}{L^2}^2 .
    \]

    As for the existence, notice that on the $1$-dimensional torus one has the inclusions $ B_{1,\infty}^\alpha\subset W^{\alpha-\epsilon, 1+\epsilon}$ and $H^1 \subset W^{-\alpha+\epsilon, \frac{1+\epsilon}{\epsilon}} = (W^{\alpha-\epsilon, 1+\epsilon})'$.
    We can then choose a sequence of smooth functions $\{u^n\}_n$ converging to $u$ in $W^{\alpha-\epsilon, 1+\epsilon}$ and the solutions to the associated equations, $\{f^n\}_n \subset L^2(0,T;W^{-\alpha+\epsilon, \frac{1+\epsilon}{\epsilon}})$.
    By considering the difference between two of such equations, with $n \neq m$, and taking the duality pairing against $f^n-f^m$, we get
    \begin{multline*}
        \frac{1}{2}\partial_t \norm{f^n-f^m}{L^2}^2 + \imath \int_\Tb (u^n-u^m) f^n \overline{(f^n-f^m)} d y + \imath \int_\Tb u^m |f^n-f^m|^2 d y \\
        + \nu \norm{\partial_y (f^n-f^m)}{L^2}^2 = 0 .
    \end{multline*}
    Analysing the real and imaginary parts of the above separately, for the former we get the equality
    \[
        \frac{1}{2}\partial_t \norm{f^n-f^m}{L^2}^2 + \nu \norm{\partial_y (f^n-f^m)}{L^2}^2
        =
        \operatorname{Im} \int_\Tb (u^n-u^m) f^n \overline{f^m} d y;
    \]
    then, integration in time yields
    \begin{align*}
        \frac{1}{2} \norm{f^n_s-f^m_s}{L^2}^2 + \nu \int_0^t \|\partial_y (f^n_s -{} & f^m_s)\|_{L^2}^2 d s
        \leq
        \int_0^t \int_\Tb |(u^n-u^m) f^n_s \overline{f^m_s}| d y d s                                                                                 \\
        \leq{}                                                                       &
        \int_0^t \norm{f^m_s}{L^\infty} \norm{u^n-u^m}{W^{\alpha-\epsilon,1+\epsilon}} \norm{f^n_s}{W^{-\alpha+\epsilon, (1+\epsilon)/\epsilon}} d s \\
        \leq{}                                                                       &
        \norm{f^m}{L^2 H^1} \norm{f^n}{L^2 H^1} \norm{u^n-u^m}{W^{\alpha-\epsilon,1+\epsilon}}
    \end{align*}
    which eventually gives the bound
    \begin{align*}
        \norm{f^n-f^m}{L^\infty L^2}^2 + 2\nu \int_0^t \|\partial_y (f^n_s - f^m_s)\|_{L^2}^2 d s
        \leq \frac{1}{\nu} \norm{f_0}{L^2}^2 \norm{u^n-u^m}{W^{\alpha-\epsilon,1+\epsilon}}
        \rightarrow 0 .
    \end{align*}
    Since a similar approach can be used for the imaginary part, we conclude that $\{f^n\}_n$ is a Cauchy sequence in $L^\infty L^2 \cap L^2 H^1$ and it converges to the solution $f$ of the limiting equation \eqref{eq:problem}.

    Uniqueness of solutions is obtained by contradiction, once one computes the difference between two solutions $f$ and $\tilde{f}$ for the same $u$.
\end{proof}

\begin{remark}
    Notice also that it can be easily proved that $f\in L^2 H^1\cap W^{1,2}H^{-1}$, hence the regularity of $f$ is in fact $C(0,T;L^2)\cap L^2(0,T;H^1)$ (see e.g. \cite[Lemma 2.1]{Bessaih1999}).
\end{remark}

\section{A parabolic estimate}
\label{app:parabolic}

While proving Lemma \ref{lemma:continuity}, we came across a bound of the form
\[
    \norm{g_0}{L^2}^2 - \norm{g_t}{L^2}^2
    \leq
    \norm{g_0}{H^1}^2 \left( \nu t + 4 |\xi| \left( \int_0^t \norm{\partial_y U(s,\cdot)}{L^2}^2 d s \right)^{\frac{1}{2}} \right) ,
\]
where $U$ is a solution of the problem
\begin{equation}\label{eq:app}
    \begin{cases}
        \partial_t U - \frac\nu2\partial_y^2 U = u - \bar u, \\
        U(0) = 0 .
    \end{cases}
\end{equation}

In order to close the proof, we then need to estimate $\int_0^t \norm{\partial_y U(s,\cdot)}{L^2}^2 d s$ in terms of the $W^{\alpha-\epsilon, 1+\epsilon}$-norm of $u$, in order to approximate it with smooth functions.

\begin{lemma}
    Let $u\in B_{1,\infty}^\alpha$, $\alpha\in(-\frac12,0)$.
    Then, if $U$ is the corresponding solution to \eqref{eq:app}, it holds
    \[
        \int_0^t \norm{\partial_y U(r,\cdot)}{L^2}^2 ds
        \leq
        C t^{\frac{1}{1+\epsilon}(\frac32+\alpha+\epsilon(\frac32+\alpha-\epsilon))} \nu^{-\frac{1}{1+\epsilon}(\frac32-\alpha+\epsilon(\frac32-\alpha+\epsilon))} \norm{u}{W^{\alpha-\epsilon,1+\epsilon}}^2
    \]
    for any $\epsilon>0$ suitably small and some constant $C>0$ only depending on $\alpha$ and $\epsilon$.
\end{lemma}

\begin{proof}
    First of all, we can test $\partial_t U - \frac{\nu}{2} \partial_y^2 U = u - \bar u$ against $U$ and get ($U$ has null mean)
    \[
        \frac12 \frac{d}{dt} \norm{U}{L^2}^2 + \frac{\nu}{2} \norm{\partial_y U }{L^2}^2
        =
        \int_\Tb u U .
    \]
    Such operation is legitimate: $u\in B^{\alpha}_{1,\infty}$ and, for all $\epsilon>0$, $B^{\alpha}_{1,\infty}\subset W^{\alpha-\epsilon,1+\epsilon}$, which is the dual of $W^{-\alpha+\epsilon, (1+\epsilon)/\epsilon}$.
    It is then not hard to check that the solution $U$ to \eqref{eq:app} belongs to $B^{\alpha+2}_{1,\infty}$, and that $B_{1,\infty}^{\alpha+2} \subset W^{-\alpha+\epsilon, (1+\epsilon)/\epsilon}$ for $\epsilon$ small enough, in particular such that
    $\epsilon < \alpha+\sqrt{\alpha^2+4\alpha+2}$.

    Now, by duality
    \[
        \frac12 \frac{d}{d t} \norm{U}{L^2}^2 + \frac{\nu}{2} \norm{\partial_y U}{L^2}^2
        =
        \int_\Tb u U
        \leq
        \norm{u}{W^{\alpha-\epsilon,1+\epsilon}} \norm{U}{W^{-\alpha+\epsilon, (1+\epsilon)/\epsilon}} ,
    \]
    and by interpolation (cf. \cite{nirenberg})
    \begin{equation}\label{eq:interp}
        \frac12 \frac{d}{d t} \norm{U}{L^2}^2 + \frac\nu2 \norm{\partial_y U}{L^2}^2
        \leq
        C_1 \norm{u}{W^{\alpha-\epsilon,1+\epsilon}} \norm{U}{L^2}^{1-\theta} \norm{\partial_y U}{L^2}^\theta ,
    \end{equation}
    where $C_1>0$ is a suitable constant,
    \[
        \theta \coloneqq \frac{1}{1+\epsilon} \left( \frac{1}{2} - \alpha + \bar{\epsilon} \right)
        \quad\text{and}\quad
        \bar\epsilon \coloneqq \epsilon\left( \frac{1}{2} - \alpha + \epsilon \right) .
    \]
    We then apply Young's inequality with $p=2/(2-\theta)$ and $q=2/\theta$:
    \begin{align*}
        \frac{d}{d t}  \norm{U}{L^2}^2 + \nu \norm{\partial_y U}{L^2}^2
         & \lesssim
        \left( \left(\frac{1}{q \nu}\right)^{\frac{1}{q}} 2 C_1 \norm{u}{W^{\alpha-\epsilon, 1+\epsilon}} \norm{U}{L^2}^{1-\theta} \right) \left( \left(q \nu\right)^{\frac{1}{q}} \norm{\partial_y U}{L^2}^{\theta} \right) \\
         & \leq
        \frac{1}{p} \left( \left(\frac{1}{q \nu}\right)^{\frac{1}{q}} 2 C_1 \norm{u}{W^{\alpha-\epsilon, 1+\epsilon}} \norm{U}{L^2}^{1-\theta} \right)^p + \nu \norm{\partial_y U}{L^2}^2                                    \\
         & =
        C_2 \nu^{-\frac{\theta}{2-\theta}} \norm{u}{W^{\alpha-\epsilon, 1+\epsilon}}^{\frac{2}{2-\theta}} \left( \norm{U}{L^2}^2 \right)^{\frac{1-\theta}{2-\theta}} + \nu \norm{\partial_y U}{L^2}^2 ,
    \end{align*}
    where $C_2 = (2-\theta) \left( 2 \theta \right)^{\frac{\theta}{2-\theta}} C_1^{\frac{2}{2-\theta}}$, yielding
    \[
        \frac{d}{d t}  \norm{U}{L^2}^2
        \lesssim
        C_2 \nu^{-\frac{\theta}{2-\theta}} \norm{u}{W^{\alpha-\epsilon, 1+\epsilon}}^{\frac{2}{2-\theta}} \left( \norm{U}{L^2}^2 \right)^{1-\frac{1}{2-\theta}}  .
    \]
    Therefore,
    \[
        \norm{U(t)}{L^2}
        \leq
        C_3 t^{1-\frac\theta2} \nu^{-\frac\theta2} \norm{u}{W^{\alpha-\epsilon, 1+\epsilon}} ,
    \]
    with $C_3 = \left( 2 \theta \right)^{\frac\theta2} C_1$.

    At this point we can integrate \eqref{eq:interp} in time and neglect the (positive) term $\norm{U(t)}{L^2}^2$ on the left-hand side:
    \begin{align*}
        \frac{\nu}{2} \int_0^t & \norm{\partial_y U(s)}{L^2}^2 ds
        \leq
        C_1 \norm{u}{W^{\alpha-\epsilon,1+\epsilon}} \int_0^t \norm{U(s)}{L^2}^{1-\theta} \norm{\partial_y U(s)}{L^2}^\theta d s                                                                                               \\
                               & \leq
        C_1 \norm{u}{W^{\alpha-\epsilon,1+\epsilon}} \left( \int_0^t \norm{U(s)}{L^2}^{\frac{2(1-\theta)}{2-\theta}} d s \right)^{\frac{2-\theta}2} \left( \int_0^t \norm{\partial_y U(s)}{L^2}^{2} d s \right)^{\frac\theta2} \\
                               & \leq
        C_4 t^{\frac{(2-\theta)^2}{2}} \nu^{-\frac\theta2(1-\theta)} \norm{u}{W^{\alpha-\epsilon,1+\epsilon}}^{2-\theta} \left( \int_0^t \norm{\partial_y U(s)}{L^2}^2 d s \right)^{\frac{\theta}{2}} ,
    \end{align*}
    where $C_4 = (2-\theta)^{-\frac{2-\theta}{2}} (2\theta)^{\frac\theta2(1-\theta)} C_1^{2-\theta}$.
    If we assume that $\partial_y U$ is not identically null on the interval $[0,t]$ we can now divide by the integral to get
    \begin{multline}\label{eq:gradineqparabolic}
        \int_0^t \norm{\partial_y U(s)}{L^2}^2 ds
        \leq
        C_5 t^{2-\theta} \nu^{-(1+\theta)} \norm{u}{W^{\alpha-\epsilon,1+\epsilon}}^2 \\
        =
        C_5 t^{\frac{1}{1+\epsilon}(\frac32+\alpha+\epsilon(\frac32+\alpha-\epsilon))} \nu^{-\frac{1}{1+\epsilon}(\frac32-\alpha+\epsilon(\frac32-\alpha+\epsilon))} \norm{u}{W^{\alpha-\epsilon,1+\epsilon}}^2 ,
    \end{multline}
    with $C_5 = (2-\theta)^{-1} 2^{\theta+1} \theta^{\frac{\theta(1-\theta)}{2-\theta}} C_1^2$ (we omit its expression in terms of $\epsilon$, for the sake of readability).
\end{proof}

In order to close the proof, we then need to estimate $\int_0^t \norm{\partial_y U(s,\cdot)}{L^2}^2 d s$ in terms of the $H^{-1}$-norm of $u$, in order to approximate it with smooth functions.

\begin{lemma}
    Let $u\in B_{1,\infty}^\alpha$, $\alpha\in(-\frac12,0)$. Then, if $U$ is the corresponding solution to \eqref{eq:app}, it holds
    \[
        \int_0^t \norm{\partial_y U(r,\cdot)}{L^2}^2 ds
        \leq
        C t^{-1} \nu^{-2} \norm{u}{H^{-1}}^2
    \]
    for some constant $C>0$.
\end{lemma}

\begin{proof}
    First of all, we can test $\partial_t U - \frac{\nu}{2} \partial_y^2 U = u - \bar u$ against $U$ and get ($U$ has null mean)
    \[
        \frac12 \frac{d}{dt} \norm{U}{L^2}^2 + \frac{\nu}{2} \norm{\partial_y U }{L^2}^2
        =
        \int_\Tb u U .
    \]
    Such operation is legitimate: $u\in B^{\alpha}_{1,\infty}$ and $B^{\alpha}_{1,\infty}\subset H^{-1}$, intended as the dual of the Sobolev space $H^1$.
    It is then not hard to check that the solution $U$ to \eqref{eq:app} belongs to $B^{\alpha+2}_{1,\infty}$, and that $B_{1,\infty}^{\alpha+2} \subset H^1$ for all $-\frac12<\alpha<0$.

    Now, by duality
    \[
        \frac12 \frac{d}{d t} \norm{U}{L^2}^2 + \frac{\nu}{2} \norm{\partial_y U}{L^2}^2
        =
        \int_\Tb u U
        \leq
        \norm{u}{H^{-1}} \norm{U}{H^1}
        \lesssim \norm{u}{H^{-1}} \norm{\partial_y U}{L^2}.
    \]
    At this point it suffices to integrate in time between $0$ and $t$, $t\in[0,T]$, and neglect the (positive) term $\norm{U(t)}{L^2}^2$, to obtain
    \begin{align*}
        \frac\nu2 \int_0^t \norm{\partial_y U(s)}{L^2}^2 d s
         & \lesssim
        \norm{u}{H^{-1}} \int_0^t \norm{\partial_y U(s)}{L^2} d s \\
         & \lesssim
        t^{\frac12} \norm{u}{H^{-1}} \left( \int_0^t \norm{\partial_y U(s)}{L^2}^2 d s \right)^{\frac12} ,
    \end{align*}
    which easily yields
    \[
        \int_0^t \norm{\partial_y U(s)}{L^2}^2 d s
        \lesssim
        t \nu^{-2} \norm{u}{H^{-1}}^2 .
    \]
\end{proof}


\section*{Acknowledgements}
The first author acknowledges the partial support of the project PNRR - M4C2 - Investimento 1.3, Partenariato Esteso PE00000013 - \emph{FAIR - Future Artificial Intelligence Research} - Spoke 1 \emph{Human-centered AI}, funded by the European Commission under the NextGeneration EU programme, of the project \emph{Noise in fluid dynamics and related models} funded by the MUR Progetti di Ricerca di Rilevante Interesse Nazionale (PRIN) Bando 2022 - grant 20222YRYSP, of the project \emph{APRISE - Analysis and Probability in Science} funded by the the University of Pisa, grant PRA\_2022\_85, and of the MUR Excellence Department Project awarded to the Department of Mathematics, University of Pisa, CUP I57G22000700001.

The second author was supported by the project \emph{Mathematical methods for climate science}, funded by the Ministry of University and Research (MUR) as part of the PON 2014-2020 ``Research and Innovation'' resources - Green Action - DM MUR 1061/2022.

\bibliographystyle{amsalpha}
\bibliography{roughmixing}
\end{document}